\documentclass[reqno]{amsart}

\usepackage{amsmath}
\usepackage{amssymb}
\usepackage{mathdots}
\usepackage[bb=boondox]{mathalfa}

\usepackage{tikz}

\newif\ifstartedinmathmode
\newcommand\encircled[1]{%
  \relax\ifmmode\startedinmathmodetrue\else\startedinmathmodefalse\fi%
  \tikz[baseline,anchor=base]{%
  \node[draw=red,circle,outer sep=0pt,inner sep=.2ex]
    {\ifstartedinmathmode$#1$\else#1\fi};}%
}
\theoremstyle{plain}
\newtheorem{theorem}{Theorem}[section]
\newtheorem{lemma}[theorem]{Lemma}
\newtheorem{corollary}[theorem]{Corollary}
\newtheorem{proposition}[theorem]{Proposition}

\theoremstyle{definition}
\newtheorem{remark}[theorem]{Remark}

\newtheorem{example}[theorem]{Example}

\numberwithin{equation}{section}

\allowdisplaybreaks[1]

\newcommand{\BC}{{\mathbb C}}
\newcommand{\BF}{{\mathbb F}}
\newcommand{\BH}{{\mathbb H}}

\newcommand{\BL}{{\mathbb L}}

\newcommand{\BR}{{\mathbb R}}

\newcommand{\cE}{{\mathcal E}}
\newcommand{\cH}{{\mathcal H}}

\newcommand{\cL}{{\mathcal L}}

\newcommand{\cS}{{\mathcal S}}

\newcommand{\fC}{{\mathfrak C}}

\newcommand{\wtilA}{\widetilde{A}}

\newcommand{\wtilK}{\widetilde{K}}\newcommand{\wtilL}{\widetilde{L}}

\newcommand{\whatA}{\widehat{A}}

\newcommand{\whatK}{\widehat{K}}\newcommand{\whatL}{\widehat{L}}

\newcommand{\al}{\alpha}
\newcommand{\be}{\beta}
\newcommand{\ga}{\gamma}

\newcommand{\la}{\lambda}\newcommand{\La}{\Lambda}

\newcommand{\rank}{\textup{rank\,}}

\newcommand{\im}{\textup{Im}\,}

\newcommand{\kr}{\textup{Ker\,}}

\newcommand{\imag}{\textup{i}\,}
\newcommand{\mat}[1]{\begin{bmatrix} #1 \end{bmatrix}}

\newcommand{\ov}[1]{{\overline{#1}}}

\newcommand{\inn}[2]{\ensuremath{\langle #1,#2 \rangle}}
\newcommand{\tu}[1]{\textup{#1}}

\newcommand{\what}[1]{{\widehat{#1}}}

\newcommand{\ands}{\quad\mbox{and}\quad}

\newcommand{\trace}{\textup{trace\,}}

\newcommand{\BBone}{\mathbb{1}}
\newcommand{\OneVec}{\vec{\mathbf{1}}}

\newcommand{\vect}{\operatorname{vec}}

\begin{document}

\title[Hill representations for $*$-linear matrix maps]{Hill representations for $*$-linear matrix maps}

\author[S. ter Horst]{S. ter Horst}
\address{S. ter Horst, Department of Mathematics, Research Focus Area:\ Pure and Applied Analytics, North-West
University, Potchefstroom, 2531 South Africa and DSI-NRF Centre of Excellence in Mathematical and Statistical Sciences (CoE-MaSS)}
\email{Sanne.TerHorst@nwu.ac.za}

\author[A. Naud\'{e}]{A. Naud\'{e}}
\address{A. Naud\'{e}, Faculty of Engineering and the Built Environment, Academic Development Unit, University of the Witwatersrand, Johannesburg, 2000 South Africa and DSI-NRF Centre of Excellence in Mathematical and Statistical Sciences (CoE-MaSS)}
\email{naudealma@gmail.com}

\thanks{This work is based on the research supported in part by the National Research Foundation of South Africa (Grant Number 118513 and 127364).}

\subjclass[2010]{Primary 15A69; Secondary 15B48}



\keywords{Linear matrix maps, Hermitian preserving maps, matricization, Choi matrix}

\begin{abstract}
In the paper \cite{H73} from 1973 R.D. Hill studied linear matrix maps $\cL:\BC^{q \times q}\to\BC^{n \times n}$ which map Hermitian matrices to Hermitian matrices, or equivalently, preserve adjoints, i.e., $\cL(V^*)=\cL(V)^*$, via representations of the form
\begin{equation*}
\cL(V)=\sum_{k,l=1}^m \BH_{kl}\, A_l V A_k^*,\quad V\in\BC^{q \times q},
\end{equation*}
for matrices $A_1,\ldots,A_m \in\BC^{n \times q}$ and continued his study of such representations in later work, sometimes with co-authors, to completely positive matrix maps and associated matrix reorderings. In this paper we expand the study of such representations, referred to as Hill representations here, in various directions. In particular, we describe which matrices $A_1,\ldots, A_m$ can appear in Hill representations (provided the number $m$ is minimal) and determine the associated Hill matrix $\BH=\left[\BH_{kl}\right]$ explicitly. Also, we describe how different Hill representations of $\cL$ (again with $m$ minimal) are related and investigate further the implication of $*$-linearity on the linear map $\cL$.
\end{abstract}

\maketitle

\section{Introduction}

In this paper we further develop the theory of a representation for linear matrix maps that was introduced and studied by R.D. Hill, in some cases with co-authors, in various papers in the 1970s and 1980s \cite{H69,H73,PH81,OH85}. Part of the study of Hill and co-authors involved a matrix reordering (see \cite{PH81,OH85}) that has reappeared recently in the study of common solutions to the Lyapunov equation \cite{BG15}, generated matrix algebras \cite{P19}, outer spectral radius and completely positive maps \cite{P19Arx} and Nevanlinna-Pick interpolation \cite{AJP20Arx}.

Throughout this paper let $\BF=\BC$ or $\BF=\BR$; we use notation as if $\BF=\BC$, so that it is clear when e.g., one requires a transpose or adjoint, and can distinguish between Hermitian and symmetric matrices, etc. Consider a linear matrix map
\begin{equation}\label{cL-Intro}
\cL:\BF^{q\times q} \to \BF^{n\times n}.
\end{equation}
Following \cite{KMcCSZ19}, we say that $\cL$ is $*$-linear if $\cL(V^*)=\cL(V)^*$ for all $V\in\BF^{q\times q}$; for $\BF=\BC$ this corresponds with $\cL$ mapping Hermitian matrices to Hermitian matrices.

In \cite{H69,H73} Hill studied representations for linear matrix maps $\cL$ of the form
\begin{equation}\label{HillRepIntro}
\cL(V)=\sum_{k,l=1}^m \BH_{kl}\, A_l V A_k^*,\quad V\in\BF^{q \times q},
\end{equation}
for matrices $A_1,\ldots,A_m\in\BF^{n \times q}$, and showed, in \cite{H73} for $\BF=\BC$, that $\cL$ is $*$-linear if and only if $\cL$ admits a representation \eqref{HillRepIntro} with $\BH:=\left[\BH_{kl}\right]_{k,l=1}^m\in\BF^{m\times m}$ Hermitian. In later work with Poluikis \cite{PH81}, Hill showed in addition that $\cL$ is completely positive if a representation as in \eqref{HillRepIntro} exists with $\BH$ positive semidefinite, connecting to the ground breaking work of Choi \cite{C75} (where such representations occur with $\BH=I_m$, where $I_p$ is the identity matrix of size $p \times p.$). We call a representation as in \eqref{HillRepIntro} a {\em Hill representation} of $\cL$ and the matrix $\BH$ defined above the associated {\em Hill matrix}. In the case of a Hill representation of $\cL$ so that the number $m$ is the smallest that can occur,  we speak of a {\em minimal Hill representation}, and in this paper we shall mostly restrict to minimal Hill representations.  Minimality of the Hill representation \eqref{HillRepIntro} implies that the matrices  $A_1,\ldots,A_m$ are linearly independent, which in turn implies, among others, that $\cL$ is $*$-linear if and only if $\BH$ is Hermitian \cite{PH81}.

In the present paper we are interested in, for instance, what matrices $A_1,\ldots,A_m$ can appear in a minimal Hill representation of a $*$-linear matrix map $\cL$ and how do different minimal Hill representations of $\cL$ relate. Moreover, the further analysis conducted in this paper plays an important role in further work on linear matrix maps for which positivity and complete positivity coincide, on which we shall report in a separate paper \cite{tHNpre}.

With the linear matrix map $\cL$ in \eqref{cL-Intro} we associate two matrices, the {\em Choi matrix} $\BL$ given by
\[
\BL=\left[\BL_{ij}\right] \in \BF^{nq \times nq},\quad \BL_{ij}=\cL\left(\cE_{ij}^{(q)}\right) \in\BF^{n \times n},
\]
where $\cE_{ij}^{(q)}$ is the standard basis element in $\BF^{q \times q}$ with a 1 on position $(i,j)$ and zeros elsewhere, and what we call the {\em matricization} of $\cL$, which is the matrix $L\in \BF^{n^2 \times q^2}$ determined by the linear map
\[
L:\BF^{q^2} \to \BF^{n^2},\quad L\,\left(\vect_{q\times q}(V)\right)= \vect_{n \times n}\left(\cL (V)\right),\quad V\in\BF^{q\times q},
\]
where $\vect_{r\times s} :\BF^{r \times s} \to \BF^{rs}$ is the vectorization operator. The matrices $\BL$ and $L$ are related through the matrix reordering of \cite{PH81,OH85} which was mentioned above and is discussed in Section \ref{S:Invol}.

We shall now describe how the minimal Hill representations of $\cL$ can be constructed. Write $L$ as a block matrix
\[
L=\mat{L_{ij}}\quad \mbox{with}\quad L_{ij} \in \BF^{n \times q}.
\]
The minimum value the number $m$ in the Hill representation of $\cL$ can attain is given by
\[
m=\rank\, \BL= \dim \tu{span}\{ L_{ij} \colon i=1,\ldots,n, \, j=1,...,q \} \subset \BF^{n\times q}.
\]
Now select any $L_1,\cdots,L_m\in \BF^{n\times q}$ so that
\begin{equation}\label{SpanL1Lm}
\tu{span}\{L_1,\ldots,L_m\}= \tu{span}\{ L_{ij} \colon i=1,\ldots,n, \, j=1,...,q \}.
\end{equation}
In particular, $L_1,\cdots,L_m$ are linearly independent in $\BF^{n\times q}$.
Thus there exist scalars $\al^{ij}_k$ and $\be_{ij}^k$ so that
\begin{equation}\label{LkLij-intro}
L_k =\sum_{i=1}^n\sum_{j=1}^q \be_{ij}^k L_{ij},\quad L_{ij}=\sum_{k=1}^m \al^{ij}_k L_k.
\end{equation}
Now, for $k=1,\dots,m$ set
\begin{equation}\label{AkBk-intro}
A_k=\left[\ov{\al}_{k}^{ij}\right]\in \BF^{n \times q} \quad \text{and} \quad B_k=\left[\be_{ij}^k\right]\in\BF^{n \times q} \quad \text{for} \quad k=1,\ldots,m,
\end{equation}
and define
\begin{equation}\label{Hill-intro}
\BH=\BH\left(\cL;L_1,\ldots,L_m\right):=\left[\OneVec_n^*\left(B_k \circ \ov{L}_l\right)\OneVec_q\right]_{k,l=1}^m\in\BF^{m \times m},
\end{equation} with $\circ$ indicating the Hadamard product and $\OneVec_p$ the all-one vector of length $p.$
In case $\cL$ is $*$-linear, $\cL$ admits a minimal Hill representation with $A_1,\ldots,A_m$ and $\BH$ as constructed above, and all minimal Hill representations of $\cL$ are of that form, as follows from the next theorem which is our first main result.

\begin{theorem}\label{T:main1}
Let $\cL$ in \eqref{cL-Intro} be a $*$-linear map with Choi matrix $\BL$ and matricization $L$. Set $m=\rank \BL$. Then $\cL$ admits a minimal Hill representation \eqref{HillRepIntro} with $A_1,\ldots,A_m$ and $\BH$ as in \eqref{AkBk-intro} and \eqref{Hill-intro}, respectively, determined by any matrices $L_1,\ldots,L_m$  which satisfy \eqref{SpanL1Lm}. Moreover, all minimal Hill representations of $\cL$ are obtained in this way. Furthermore, the matrices $A_1,\ldots,A_m$ satisfy
\begin{equation}\label{A1Amcon}
\tu{span}\{A_1,\ldots,A_m\}= \tu{span}\{ L_{ij} \colon i=1,\ldots,n, \, j=1,...,q \},
\end{equation}
and conversely, any matrices $A_1,\ldots,A_k$ in $\BF^{n \times q}$ with this property appear in some Hill representation of $\cL$.
\end{theorem}

Theorem \ref{T:main1} will be proved in Section \ref{S:HillReps} and characterises precisely the matrices $A_1,\ldots,A_m$ that can appear in a minimal Hill representation of a $*$-linear matrix map $\cL$ and provides a construction starting from linearly independent matrices $L_1,\ldots,L_m$ satisfying \eqref{SpanL1Lm}. Conversely, given any $A_1,\ldots,A_m$ satisfying \eqref{A1Amcon}, it is also possible to construct the matrices $L_1,\ldots,L_m$ (and $\BH$) satisfying \eqref{SpanL1Lm} so that $A_1,\ldots,A_m$ appear out of the above construction starting from $L_1,\ldots,L_m$; see Theorem \ref{T:HillRepChar} below.

The characterization of the minimal Hill representations of a $*$-linear map and associated formulas determined in Section \ref{S:HillReps} also make it possible to compare minimal Hill representations. The following theorem, which is our second main result, provides a summary of the result obtained in Section \ref{S:HillReps} in this regard.

\begin{theorem}\label{T:main2}
Let $\cL$ in \eqref{cL-Intro} be a $*$-linear map with Choi matrix $\BL$ and matricization $L$. Set $m=\rank \BL$. Let $L_1,\ldots,L_m$ and $L_1',\ldots,L_m'$ be two sets of matrices for which the spans coincides the span of the blocks entries $L_{ij}$ of $L$. Define $A_1,\ldots,A_k$ and $\BH$  as in \eqref{AkBk-intro} and \eqref{Hill-intro}, respectively, and define $A_1',\ldots,A_k'$ and $\BH'$ correspondingly for the matrices $L_1',\ldots,L_m'$. Then there exists an invertible matrix $\Phi$ in $\BF^{m \times m}$ so that
\begin{equation}\label{LL'AA'forms}
\mat{L_1\\ \vdots \\ L_m}= \left(\Phi \otimes I_n\right) \mat{L_1'\\ \vdots \\ L_m'},\quad \left(\Phi^* \otimes I_n\right) \mat{A_1\\ \vdots \\ A_m}= \mat{A_1'\\ \vdots \\ A_m'},\quad \BH=\Phi \BH'\Phi^*,
\end{equation} with $\otimes$ the Kronecker product.
\end{theorem}

We shall also prove Theorem \ref{T:main2} in Section \ref{S:HillReps}, together with several other identities that link the matrices associated with two minimal Hill representations of $\cL$. We further point out that the invertible matrix $\Phi$ in Theorem \ref{T:main2} can also be expressed explicitly in terms of the matrices associated with the sets $L_1,\ldots,L_m$ and $L_1',\ldots,L_m'$, namely as in \eqref{PhiXi}.

Finally, while $*$-linearity is easy to characterise in terms of the Choi matrix, that is, $\cL$ is $*$-linear if and only if $\BL$ is Hermitian, the implication of the characterisation in terms of the matricization $L$ is less straightforward. In \cite{PH81} Poluikis and Hill gave a characterisation in terms of the entries of $L$, see Theorem \ref{T:*-linear} and Proposition \ref{P:Herm} below, which we exploit further in several special cases. In particular, the characterisation of Poluikis and Hill implies that many structural properties (e.g., zero patterns, Toeplitz structure, Hankel structure, etc) occur in $L$ at the level of a block matrix, $L=\mat{L_{ij}}$, if and only if they occur at the level of the blocks $L_{ij}$.

The paper is structured as follows. In addition to the current introduction there are five sections. In Section \ref{S:Pre} we provide various preliminaries from matrix analysis used throughout the paper. Sections \ref{S:Invol} and \ref{S:Linear} contain our discussions on the matrix reordering from \cite{PH81,OH85} and representations of linear and $*$-linear maps. These sections are partially a literature review, but also contain various extensions of known results as well as a few new results. The main contributions of this paper are in Section \ref{S:HillReps} which contains the proofs of our main results as well as many other results on minimal Hill representations. In the final section, Section \ref{S:Example} we further explore the implication of $*$-linearity on the structure of the matricization $L$ and determine the Hill matrix $\BH$ more explicitly in case the matrices $L_1,\ldots,L_m$ are selected among the block entries $L_{ij}$ from $L$.

\section{Preliminaries}\label{S:Pre}

In this section we collect some of the notation and elementary matrix analysis results used throughout the paper. The notation is mostly standard and the formulas presented here can be easily verified and appear in most advanced linear algebra textbooks, cf., \cite{HJ85,HJ91}. Throughout $\BF=\BC$ or $\BF=\BR$. We write $\BF^{n \times m}$ for the vector space of $n \times m$ matrices over $\BF$ and $\BF^n$ for the space of all (column) vectors over $\BF$ of length $n$. Occasionally we will identify $\BF^n$ with $\BF^{n \times 1}$, so that matrix operations can be applied to vectors in $\BF^n$. The orthogonal complement of some subset $W\subset \BF^n$, with respect to the Euclidean inner product, is indicated with $W^\perp.$

The standard $j$-th basis element in $\BF^n$ is denoted by $e_j^{(n)}$ or simply $e_j$ when the length is clear from the context. We write $\cE_{ij}^{(n,m)}$ for the standard basis element of $\BF^{n \times m}$ with $1$ on position $(i,j)$ and zeros elsewhere, i.e., $\cE_{ij}^{(n,m)}=e_i^{(n)}e_j^{(m)T}$, abbreviated to $\cE_{ij}^{(n)}$ when $m=n$. With $\OneVec_n$ we indicate the all-one vector of length $n$ and with $\BBone_{n \times m}$ the all-one matrix of size $n \times m$, so that $\BBone_{n \times m}=\OneVec_{n}\OneVec_m^T$. Also here, we write $\BBone_n$ for $\BBone_{n \times n}$. Furthermore, $I_n$ denotes the $n \times n$ identity matrix and $P^{(n)}_{i,j}$ the permutation matrix of size $n \times n$ that interchanges the $i$-th and $j$-th row/ column, abbreviated to $P_{i,j}$ when there can be no confusion about the size.

For $A \in \BF^{n \times m}$ we write $A^T$ for its transpose, $A^*$ for its adjoint, $\overline{A}$ for its complex conjugate, $\kr{A}$ for its nullspace and $\im A$ for its range. We write $\cH_n$ for the $n \times n$ Hermitian matrices and $\cS_n$ for the $n \times n$ symmetric matrices. For $\BF=\BR$, of course, $\cH_n$ and $\cS_n$ coincide. With $A\geq 0$ (resp.\ $A>0$) we indicate that $A$ is positive semidefinite (resp.\ positive definite).

The \emph{vectorization} of a matrix $T \in \BF^{n \times m}$ is the vector $\vect_{n \times m }{(T)} \in \BF^{nm}$ defined as \[\vect_{n \times m}(T)= \sum_{j=1}^m\left(e_j^{(m)} \otimes I_n\right)Te_j^{(m)}.\]
Note that the vectorization operator $\vect_{n \times m}$ defines an invertible linear map from $\BF^{n\times m}$ onto $\BF^{nm}$ whose inverse is given by
\begin{equation}\label{vecinv}
\vect_{n \times m}^{-1}:\BF^{nm}\to \BF^{n\times m},\quad \vect_{n \times m}^{-1}(x)=\sum_{j=1}^m\left(e^{(m)^T}_j \otimes I_n\right) x ,\quad x\in\BF^{nm}.
\end{equation} If $m=n$ we just write $\vect_n$ and $\vect_n^{-1}$ and if the sizes are clear from the context, the indices are often left out. Furthermore, the {\em trace} of a square matrix $A$ in $\BF^{n \times n}$ is denoted as $\trace(A)$. Note that $\trace$ is a linear map from $\BF^{n \times n}$ into $\BF$  which satisfies $\trace(A^T)=\trace(A)$, for $A\in \BF^{n \times n}$, and $\trace(AB)=\trace(BA)$, for $A,B^T\in\BF^{n \times m}$. Moreover, $\BF^{n \times m}$ becomes an inner product space via the trace inner product given by
\begin{equation}\label{traceinner}
\inn{A}{B}_{\BF^{n \times m}}=\trace(AB^*)=\inn{\vect_{n\times m}(A)}{\vect_{n\times m}(B)}_{\BF^{nm}}.
\end{equation}

\noindent The {\em Kronecker product} of matrices $A=\left[a_{ij}\right]\in\BF^{n \times m}$ and $B\in\BF^{k \times l}$ is defined as
\[
A\otimes B=\left[a_{ij}B\right]\in \BF^{(n k) \times (m l)}.
\]
Note that
\[
\left(A \otimes B\right)\left(C \otimes D \right)=(AC) \otimes (BD)
\]
holds for all matrices of appropriate size. Furthermore, we have the identity
\[
\vect\left(A XB^T\right)=(B \otimes A)\vect(X)
\]
which, when $A$ and $B$ are taken to be (transposes of) vectors yields
\begin{equation}\label{VectForm1}
(z \otimes x)^T \vect_{n \times m}(W)=x^T W z, \quad W \in \BF^{n \times m}, x \in \BF^n, z \in \BF^m.
\end{equation}
Next note that
\[
\vect_{n \times m}\left(v w^T\right)=w \otimes v,\quad v\in\BF^n,w\in\BF^m,
\]
from which we obtain that
\begin{equation}\label{UnitVectId1}
\vect_{m \times n}\left(e_l^{(m)}e_k^{(n)^T}\right)= e_k^{(n)} \otimes e_l^{(m)} =e_{(k-1)m +l}^{(nm)}.
\end{equation}
Next, recall that the {\em Hadamard product} of matrices $A=\left[a_{ij}\right], B=\left[b_{ij}\right]\in\BF^{n \times m}$ is defined as
\[
A\circ B=\left[a_{ij}b_{ij}\right]\in\BF^{n \times m}.
\]
Then we have
\begin{equation}\label{HadTens}
 \left(A\otimes B\right) \circ \left(C \otimes D\right)=\left(A \circ C\right) \otimes \left(B\circ D\right).
\end{equation}
Moreover, the trace inner product on $\BF^{n \times m}$ can also be expressed as
\begin{equation}\label{TraceInn}
\trace\left(A B^*\right)=\left\langle A,B\right \rangle_{\BF^{n \times m}} = \OneVec_n^T \left(A \circ \ov{B}\right)\OneVec_m.
\end{equation}
We also point out here the formula for $C=\left[\ga_{ij}\right]\in\BF^{m \times m}$ and $V_{ij}\in\BF^{r \times r}$:
\begin{equation}\label{SumCirc}
\sum_{i,j=1}^m \ga_{ij} V_{ij}= \left(\OneVec_m \otimes I_r\right)^* \left(\left(C \otimes \BBone_r\right) \circ \left[V_{ij}\right]_{i,j=1}^m\right) \left(\OneVec_m \otimes I_r\right),
\end{equation}
which can easily be verified and will be of use in the sequel.

Finally, we define the {\em canonical shuffle} $\fC_n :\BF^{n^2} \to \BF^{n^2}$ which on pure tensors is given by
\[
\fC_n (z \otimes x) = x \otimes z,
\]
and extended to $\BF^{n^2}$ by linearity. The canonical shuffle can also be defined on tensors of vectors of different size, but we will not need that here.  Note that $\fC_n$ is a linear self-invertible map on $\BF^{n^2}$ which also satisfies $\fC_n^*=\fC_n$. Hence the matrix corresponding to $\fC_n$ is a signature matrix (selfadjoint and unitary).

\section{The matrix reordering $\La_{(n,q)}^{(p,r)}$}\label{S:Invol}

In this section we study the following linear map
\begin{align}
&\La_{(n,q)}^{(p,r)}:\BF^{np \times qr}\to \BF^{nq \times pr} , \quad  \La_{(n,q)}^{(r,p)}(S)\left(e_j^{(r)} \otimes e_i^{(p)}\right) =\vect_{n \times q}\left(S_{ij}\right), \label{Lambda}\\
&\mbox{ where }S =\left[S_{ij}\right]\in \BF^{np\times qr},  \mbox{ with } S_{ij}\in\BF^{n\times q} \text{ for } \ 1\le i\le p \text{ and } 1 \le j \le r.\notag
\end{align}
Hence, if $R=\La_{(n,q)}^{(p,r)}(S)\in \BF^{nq \times pr}$, then, using \eqref{UnitVectId1}, we find that
\[
R=\mat{w_1&\cdots& w_{pr}}\quad \mbox{with}\quad w_{(j-1)p+i}= \vect_{n \times q}\left(S_{ij}\right),\ 1 \le i \le p, \quad 1 \le  j \le r.
\]
To the best of our knowledge this map first appeared in the work of Poluikis and Hill \cite{PH81}, with $n=q$ and $p=r$, in their study of completely positive maps, and was later studied together with several other matrix reorderings by Oxenrider and Hill \cite{OH85}, again for $n=q$ and $p=r$. More recently, for the case where $n=p=q=r$, this map appeared in \cite{BG15} in the study of common solutions to Lyapunov equations and it was rediscovered, also with $n=p=q=r$, by Pascoe in \cite{P19}, where it was used to study matrix algebras, and used subsequently in \cite{P19Arx,AJP20Arx}.

Several of the results obtained here have appeared in the literature mentioned above, for the special cases studied there. For the reader's convenience we give proofs of all the results, as they are not very long and often different from the proofs given elsewhere, and indicate where the original results appeared.

It is clear that $\La_{(n,q)}^{(p,r)}$ is a bijective map, with inverse given by
\begin{align}
\left(\La_{(n,q)
 }^{(p,r)}\right)^{-1}(R) & = S \in \BF^{np \times qr}\quad \text{where} \quad S=[S_{ij}] \ \  \mbox{with} \label{LambdaInv}\\ S_{ij}=\vect_{n \times q}^{-1}&\left(R \left(e_j^{(r)} \otimes e_i^{(p)}\right)\right) \in \BF^{n \times q} \text{ and }   R\in \BF^{nq \times pr},\ 1 \le i \le p, \, 1 \le j \le r.\notag
\end{align}
What is less straightforward is that the inverse map is of the same form. In fact, when $n=p=q=r$ then $\La_{(n,n)}^{(n,n)}$ is an involution, i.e., also self-inversive, as was noted in \cite[Page 210]{OH85} and rediscovered in \cite{P19}. The general formula is given in the next proposition; beyond the case $n=p=q=r$ this result does not seem to appear in the literature.

\begin{proposition}
For $ \La^{(p,r)}_{(n,q)}$ defined as in \eqref{Lambda} it follows that $\left( \La^{(p,r)}_{(n,q)}\right)^{-1}= \La^{(q,r)}_{(n,p)}.$
\end{proposition}

\begin{proof}
[\bf Proof] Take $S=\begin{bmatrix}S_{ij} \end{bmatrix}\in \BF^{np \times qr}$ with $S_{ij} \in \BF^{n \times q}$ and $1 \le i \le p,$ $1 \le j \le r.$ Then \begin{equation*}
\begin{aligned}
R:=\La^{(p,r)}_{(n,q)}(S) &= \begin{bmatrix} S_{11}e^{(q)}_1 & \hdots & S_{p1}e^{(q)}_1 & \hdots &
S_{1r}e^{(q)}_1 & \hdots & S_{pr}e^{(q)}_1 \\ \vdots &  & \vdots &  & \vdots &  & \vdots \\ S_{11}e^{(q)}_q & \hdots & S_{p1}e^{(q)}_q & \hdots &
S_{1r}e^{(q)}_q & \hdots & S_{pr}e^{(q)}_q \end{bmatrix} \end{aligned}.
\end{equation*}
Hence $R=[R_{ij}]\in \BF^{nq \times pr}$ with $R_{ij}\in \BF^{n \times p}$ given by $R_{ij}=\mat{S_{1j}e^{(q)}_i & \hdots & S_{pj}e^{(q)}_i}$.
Then
\[
Q:= \Lambda_{(n,p)}^{(q,r)}\left(\La^{(p,r)}_{(n,q)}(S)\right)=\Lambda_{(n,p)}^{(q,r)}\left(R\right)
\]
has the form $Q=[Q_{ij}]\in \BF^{np \times qr}$ with $Q_{ij}\in \BF^{n \times q}$ given by
\begin{align*}
Q_{ij} &=\mat{R_{1j}e^{(p)}_i & \hdots & R_{qj}e^{(p)}_i}
=\mat{S_{ij}e^{(q)}_1 & \hdots & S_{ij}e^{(q)}_q}=S_{ij}.
\end{align*}
Hence $S=\Lambda_{(n,p)}^{(q,r)}\left(\La^{(p,r)}_{(n,q)}(S)\right)$ for all $S\in \BF^{np \times qr}$.
\end{proof}

Thus the inverse of the bijective map $\La^{(p,r)}_{(n,q)}$ is also given by \begin{equation*}
    \begin{aligned}
    \Lambda_{(n,p)}^{(q,r)}(R)&=S \in \BF^{np \times qr} \text{ where } R=\begin{bmatrix}R_{kj} \end{bmatrix}\in \BF^{nq \times pr}, \text{ with } \\    R_{kj}&=\vect_{n \times p}^{-1}\left(S\left(e_j^{(r)} \otimes e_{k}^{(q)}\right)\right)\in \BF^{n \times p} \text{ for } 1\le k \le q \text{ and } 1 \le j \le r.
    \end{aligned}
\end{equation*}

\begin{lemma}\label{L:DiadTens}
The following identities hold:
\begin{equation}\label{InvolId1}
\begin{aligned}
\La_{(n,q)}^{(p,r)}\left(\vect_{n\times p}(B)\vect_{q \times r}(A)^T\right)&=A\otimes B,\quad A\in\BF^{q \times r}, \,B\in \BF^{n\times p};\\
\La_{(n,q)}^{(p,r)}\left(A\otimes B\right)=\vect_{n \times q}(B)&\vect_{p \times r}(A)^T,\quad A\in \BF^{p\times r}, \, B \in \BF^{n \times q}.
\end{aligned}
\end{equation}
\end{lemma}

The first of the identities in \eqref{InvolId1} was proved in \cite[Theorem 3]{OH85}, while the second was proved in \cite{P19} for the case $n=p=q=r$, so that it also follows from the first and the fact that $\La^{(n,n)}_{(n,n)}$ is an involution.

\begin{proof}[\bf Proof]
It suffices to check the identities on vectors of the form $e_j\otimes e_i=e_j^{(r)} \otimes e_i^{(p)}$. For the first identity, set $S=\vect_{n \times p}(B)\vect_{q \times r}(A)^T$ for $A\in\BF^{q \times r}$ and $B\in \BF^{n\times p}$ and write
\[
A=\mat{A_1&\cdots&A_r},\ A_i\in\BF^q,\quad B=\mat{B_1&\cdots&B_p},\ B_j\in\BF^n.
\]
Then $S_{ij}=B_i A_j^T$ and thus
\[
\La_{(n,q)}^{(p,r)}(S) \left(e_j \otimes e_i\right)
=\vect_{n \times q}\left(B_i A_j^T\right) = A_j \otimes B_i = \left(A\otimes B\right)  \left(e_j \otimes e_i\right).
\]

For the second identity, take $S=A\otimes B$ for $A\in \BF^{p\times r}$ and $B \in \BF^{n \times q}$. Then $S_{ij}=a_{ij}B$, and thus
\begin{align*}
\La_{(n,q)}^{(p,r)}(S) \left(e_j \otimes e_i\right) &= \vect_{n \times q}\left(a_{ij}B\right) =\vect_{n \times q}(B) a_{ij} \\ &=\vect_{n \times q}(B)\vect_{p \times r}(A)^T \left(e_j \otimes e_i\right).\qedhere
\end{align*}
\end{proof}

The identity obtained in the following lemma does not seem to appear in the literature, although it resembles an identity in \cite[Proposition 2.1]{P19Arx}, from which it can be proved, at least for the case $n=p=q=r$.

\begin{lemma}\label{L:InvolId}
For $S\in\BF^{np \times qr}$ we have
\begin{align}\label{InvolId3}
  \left(z^\prime\otimes x^\prime\right)^T \La_{(n,q)}^{(p,r)}(S) \left(z \otimes x\right) & = \left(x\otimes x^\prime\right)^T S \left(z \otimes z^\prime\right),\\ \text{for } z&\in \BF^r, \, z^\prime\in\BF^q, \, x \in \BF^p,\,x^{\prime} \in \BF^{n}. \nonumber
\end{align}
In particular, we have
\begin{align}\label{InvolId4}
 \left(z^\prime \otimes I_n\right)^T \La_{(n,q)}^{(p,r)}(S) \left(I_r \otimes x\right) & = \left(x\otimes I_n\right)^T S \left(I_r \otimes z^\prime\right),\quad z^\prime\in\BF^q \text{ and } x \in \BF^p.
\end{align}
\end{lemma}

\begin{proof}[\bf Proof]
Let $S=[S_{ij}]\in\BF^{np \times qr}$ with $S_{ij}\in\BF^{n \times q}$, where $1 \le i \le p$, $1 \le j \le r$, and $z \in \BF^{r}$, $z^\prime\in\BF^q$, $x\in\BF^{p}$, $x^\prime \in \BF^n$. Then
\begin{align*}
\La_{(n,q)}^{(p,r)}(S)(z\otimes x) & = \sum_{i=1}^p\sum_{j=1}^r x_i z_j \vect_{n\times q}\left(S_{ij}\right)= \vect_{n \times q} \left( \sum_{i=1}^p\sum_{j=1}^r x_i z_j S_{ij} \right)\\
&=\vect_{n \times q}\left(\left(x\otimes I_n\right)^T S \left(z\otimes I_q\right)\right).
\end{align*}
Using \eqref{VectForm1} we obtain
\begin{align*}
\left(z^\prime\otimes x^\prime\right)^T\La_{(n,q)}^{(p,r)}(S)(z\otimes x) & = \left(z^\prime\otimes x^\prime\right)^T \vect_{n \times q}\left(\left(x\otimes I_n\right)^T S \left(z\otimes I_q\right)\right)\\
&= \left(x\otimes x^\prime \right)^T S \left(z\otimes z^\prime \right).
\end{align*}
Hence \eqref{InvolId3} holds. Then \eqref{InvolId4} follows since $x^\prime \in\BF^n$ and $z \in \BF^r$ were taken arbitrarily.
\end{proof}

In terms of the scalar entries, the relation between $S$ and $\La_{(n,q)}^{(p,r)}(S)$ is as described in the following corollary; for $n=q$ and $p=r$ it follows directly from Theorem 1 in \cite{OH85}.

\begin{corollary}\label{C:EntryDescr}
Let $S\in\BF^{np \times qr}$ and $R \in \BF^{nq \times pr}$ with $S=\left[S_{ij}\right]$ and $R=\left[R_{lj}\right]$, where $S_{ij}=\left[v_{kl}^{ij}\right]\in\BF^{n \times q}$ and $R_{lj}=\left[w_{ki}^{lj}\right] \in\BF^{n \times p}$ for $1 \le i\le p$, $1 \le j \le r$. Then
\[
R=\La_{(n,q)}^{(p,r)}(S)\quad \Longleftrightarrow \quad
v_{kl}^{ij}=w_{ki}^{lj},\ \ 1 \le i \le p, \, 1 \le j \le r, \, 1 \le k \le n, \, 1 \le l \le q.
\]
\end{corollary}

\begin{proof}[\bf Proof]
For $1 \le i\le p$, $1 \le j \le r$, $1 \le k \le n$ and $1 \le l \le q$ we have
\begin{equation*} \begin{aligned}
v_{kl}^{ij}=\left(e_k^{(n)}\right)^T S_{ij}e_l^{(q)}&=\left(e_i^{(p)} \otimes e_k^{(n)}\right)^T S\left(e_j^{(r)}\otimes e_l^{(q)}\right) \\
&= \left(e_l^{(q)} \otimes e_k^{(n)}\right)^T \La_{(n,q)}^{(p,r)}(S)\left(e_j^{(r)}\otimes e_i^{(p)}\right),
\end{aligned} \end{equation*}
and likewise $w_{{ki}}^{lj}=\left(e_l^{(q)} \otimes e_k^{(n)}\right)^T R\left(e_j^{(r)}\otimes e_i^{(p)}\right)$. Thus $v_{kl}^{ij}= w_{{ki}}^{lj}$ is the same as
\[
\left(e_l^{(q)} \otimes e_k^{(n)}\right)^T \La_{(n,q)}^{(p,r)}(S)\left(e_j^{(r)}\otimes e_i^{(p)}\right)=\left(e_l^{(q)} \otimes e_k^{(n)}\right)^T R\left(e_j^{(r)}\otimes e_i^{(p)}\right).
\]
The result follows by varying $i,j,k,l$.
\end{proof}

As a consequence, for the case where $p=n$ and $r=q$, we obtain the following characterization of when $\La_{(n,q)}^{(n,q)}(S)\in\cH_{nq}$; see Theorem 2 in \cite{PH81} for the characterization in terms of the matrix entries.

\begin{proposition}\label{P:Herm}
Let $S\in\BF^{n^2 \times q^2}$ with $S=\left[S_{ij}\right]$ where $S_{ij}=\left[v_{kl}^{ij}\right] \in \BF^{n \times q}$ for $v_{kl}^{ij}\in\BF$, $1 \le i,k \le n$ and $1 \le j,l \le q$. Then $\La_{(n,q)}^{(n,q)}(S)\in\cH_{nq}$ if and only if
\begin{equation}\label{HermChar1}
\ov{S}=\fC_{n} S \fC_q,
\end{equation}
or, equivalently,
\begin{equation}\label{HermChar2}
v_{kl}^{ij}=\ov{v}_{ij}^{kl},\quad 1 \le i,k \le n \quad  \text{and} \quad 1 \le j,l \le q.
\end{equation}
\end{proposition}

\begin{proof}[\bf Proof]
For $R=\La_{(n,q)}^{(n,q)}(S)$ as in Corollary \ref{C:EntryDescr} we have $R=R^*$ if and only $w_{ik}^{jl}=\ov{w}_{ki}^{lj}$. Via the entrywise characterization of the relation $R=\La_{(n,q)}^{(n,q)}(S)$  in Corollary \ref{C:EntryDescr} it follows that $\La_{(n,q)}^{(n,q)}(S)=\La_{(n,q)}^{(n,q)}(S)^*$ if and only if
\[
v_{kl}^{ij}= w_{ki}^{lj}=\ov{w}_{ik}^{jl}=\ov{v}_{ij}^{kl} \quad
\text{for all} \quad 1 \le i,k \le n \quad  \text{and} \quad 1 \le j,l \le q.
\]
Hence we see that $\La_{(n,q)}^{(n,q)}(S)\in\cH_{nq}$ is equivalent to \eqref{HermChar2}. It remains to prove the equivalence of \eqref{HermChar1} and \eqref{HermChar2}. Note that \eqref{HermChar2} can be rewritten as
\begin{equation}\label{HermChar3}
\left(e_i^{(n)} \otimes e_k^{(n)}\right)^T \!\! S \!\left(e_j^{(q)}\otimes e_l^{(q)}\right)\! =\! v_{kl}^{ij}\! =\ov{v}_{ij}^{kl}\! =\! \left(e_k^{(n)} \otimes e_i^{(n)}\right)^T\ov{S}\left(e_l^{(q)}\otimes e_j^{(q)}\right).
\end{equation}
Writing $x,z\in\BF^q$ and $x^\prime,z^\prime \in \BF^n$ as linear combinations of its basis vectors, it follows that \eqref{HermChar3} implies
\begin{equation}\label{HermChar4}
\left(z^\prime \otimes x^\prime \right)^T S\left(z\otimes x\right) = \left(x^\prime \otimes z^\prime\right)^T\ov{S}\left(x\otimes z\right),\quad x,z\in\BF^q \quad \text{and} \quad x^\prime,z^\prime \in \BF^{n},
\end{equation}
while, conversely, \eqref{HermChar3} follows from \eqref{HermChar4} by specifying $x,x^\prime,z,z^\prime$ as basis vectors. Note that the right hand side in \eqref{HermChar4} can be rewritten as
\begin{align*}
 \left(x^\prime \otimes z^\prime\right)^T\ov{S}(x\otimes z)
 &  = \left(z^\prime \otimes x^\prime\right)^T\fC_n \ov{S} \fC_q (z\otimes x).
\end{align*}
Since all vectors in $\BF^{n^2}=\BF^n \otimes \BF^n$ and $\BF^{q^2}=\BF^q \otimes \BF^q$ can be written as sums of pure tensors, it follows that \eqref{HermChar1} and \eqref{HermChar2} are equivalent.
\end{proof}

The next lemma explains how $\La_{(n,q)}^{(p,r)}$ behaves with respect to interchanging of (block) rows and (block) columns. This result also follows directly from Lemma \ref{L:InvolId}; see also the 4-modularity property in \cite[Proposition 2.1]{P19Arx}. Recall that $P_{i,j}$ is the permutation matrix that interchanges the $i$-th and $j$-th row/column, to be interpreted as the identity matrix in case $i=j$.

\begin{lemma}\label{L:BlockPerms}
For $S\in\BF^{np \times qr}$ and $1\leq i_1,i_2 \leq n$, $1\leq j_1,j_2 \leq q$, $1\leq k_1,k_2 \leq p$, $1\leq l_1,l_2 \leq r$ we have
\[
\La_{(n,q)}^{(p,r)}\left(\left(P_{i_1,i_2} \otimes P_{k_1,k_2}\right) S \left( P_{j_1,j_2} \otimes P_{l_1,l_2}\right)\right)
= \left(P_{i_1,i_2} \otimes P_{j_1,j_2}\right) \La_{(n,q)}^{(p,r)}(S) \left(P_{k_1,k_2} \otimes P_{l_1,l_2}\right).
\]
\end{lemma}

The next corollary follows immediately from the above lemma.

\begin{corollary}\label{C:BlockPerms}
Let $S\in\BF^{n^2 \times q^2}$. For  $1\leq i_1,i_2\leq n$, $1\leq j_1,j_2\leq q$ set
\[
R= \La_{(n,q)}^{(n,q)}(S) \ands
R'= \La_{(n,q)}^{(n,q)} \left(\left(P_{i_1,i_2} \otimes P_{i_1,i_2}) S (P_{j_1,j_2} \otimes P_{j_1,j_2}\right)\right).
\]
Then $R \in \cH_{nq}$ if and only if $R' \in \cH_{nq}$. Furthermore, $R \geq 0$ if and only if $R' \geq 0$.
\end{corollary}

\section{Linear and $*$-linear matrix maps}\label{S:Linear}

Consider a linear matrix map $\cL$ of the form
\begin{equation}\label{cL-NS}
\cL:\BF^{q\times r} \to \BF^{n\times p}.
\end{equation}
We associated two matrices with $\cL$ that can be used to represent $\cL$ and describe its properties.\medskip

\paragraph{\bf Matricization}
Since $\cL$ is a linear map from $\BF^{q \times r}$ to $\BF^{n\times p}$, via the vectorization operator we can translate $\cL$ to a linear map from $\BF^{qr}$ to $\BF^{np}$:
\begin{equation}\label{Ldef}
L:\BF^{qr} \to \BF^{np},\quad L\,\left(\vect_{q\times r}(V)\right)= \vect_{n \times p}\left(\cL (V)\right),\quad V\in\BF^{q\times r}.
\end{equation}
In the usual way we identify $L$ with a matrix in $\BF^{np \times qr}$ with respect to the standard bases in $\BF^{qr}$ and $\BF^{np}$, which we call the {\em matricization} of $\cL$. Since $\cL$ and $L$ are similar, it follows that they have the same eigenvalues, and hence $\cL$ is bijective if and only if $L$ is invertible.

Since vectorization is an invertible operation, we can describe $\cL$ in terms of $L$ as
\begin{equation}\label{Linv}
\cL(V)= \vect_{n \times p}^{-1}\left(L\vect_{q \times r}(V)\right),\quad V \in\BF^{q\times r}.
\end{equation}
This also implies that any $L\in\BF^{np \times qr}$ defines a linear map $\cL$ as in \eqref{cL-NS} via \eqref{Linv}.\medskip

\paragraph{\bf The Choi matrix representation}
The {\em Choi matrix} $\BL$ associated with $\cL$ is obtained by considering the action of $\cL$ on the standard basis elements $\cE_{ij}^{(q,r)}$ in $\BF^{q \times r}$:
\begin{equation}\label{BLdef}
\BL=\left[\BL_{ij}\right] \in \BF^{nq \times pr},\quad \BL_{ij}=\cL\left(\cE_{ij}^{(q,r)}\right) \in\BF^{n \times p},
\end{equation}
where $i=1,\ldots,q$ and $j=1, \ldots, r.$ For $V\in\BF^{q \times r}$ with matrix representation $V=\left[v_{ij}\right]$, we have $V=\sum_{i=1}^q \sum_{j=1}^r v_{ij} \cE_{ij}^{(q,r)}$. Hence, via the linearity of $\cL$ we have the following representation of $\cL$ in terms of the Choi matrix:
\begin{equation}\label{BL to cL1}
\cL(V)=\sum_{i=1}^q \sum_{j=1}^r v_{ij}\BL_{ij},\quad V=[v_{ij}]\in\BF^{q\times r}.
\end{equation}
This can also be written as
\begin{equation}\label{BL to cL2}
\cL(V)=\left(\OneVec_q\otimes I_n\right)^T\left(\BL \circ \left(V \otimes \BBone_{n \times p}\right)\right)\left(\OneVec_r\otimes I_p\right).
\end{equation}
Also here, any $\BL\in\BF^{nq \times pr}$ defines a linear map $\cL$ as in \eqref{cL-NS}, here via \eqref{BL to cL1} or \eqref{BL to cL2}.

The next proposition shows how $\BL$ can be obtained from $L$, and conversely $L$ from $\BL$. This observation was made in \cite{PH81} (for $n=q$ and $p=r$).

\begin{proposition}\label{P:L-BLrel}
Consider a linear map $\cL$ as in \eqref{cL-NS} and define $L$ by \eqref{Ldef} and $\BL$ by \eqref{BLdef}. Then $\BL=\La_{(n,q)}^{(p,r)}(L)$, or, equivalently, $L=\La_{(n,p)}^{(q,r)}(\BL)$, where $\La_{(n,q)}^{(p,r)}$ is the matrix map defined in \eqref{Lambda}.
\end{proposition}

\begin{proof}[\bf Proof]
For $i=1,\ldots,q$ and $j=1, \ldots, r$ we have
\[
\BL_{ij}=\cL\left(\cE_{ij}^{(q,r)}\right)=\vect^{-1}_{n \times p}\left(L \vect_{q \times r}\left(\cE_{ij}^{(q,r)}\right)\right)= \vect_{n \times p}^{-1}\left(L \left(e_j^{(r)}\otimes  e_i^{(q)}\right)\right).
\]
Hence $L \left(e_j^{(r)}\otimes  e_i^{(q)}\right)=\vect_{n \times p}\left(\BL_{ij}\right)$ for all $i=1,\ldots,q$ and $j=1, \ldots, r$ which proves that $\La_{(n,p)}^{(q,r)}\left(\BL\right)=L$.
\end{proof}

In the remainder of the paper we only consider linear maps on square matrices. Hence we will look at linear maps of the form
\begin{equation}\label{cL}
\cL:\BF^{q\times q} \to \BF^{n\times n}.
\end{equation}
In this case, following \cite{KMcCSZ19}, the linear matrix map $\cL$ in \eqref{cL} is called {\em $*$-linear} in case it respects adjoints, i.e., if $\cL(V^*)=\cL(V)^*$ for all $V\in\BF^{q\times q}$. We say that $\cL$ is {\em Hermitian-preserving} if $\cL$ maps $\cH_q$ into $\cH_n$. Clearly, a $*$-linear map is Hermitian-preserving. For $\BF=\BC$ the converse is also true.

\begin{theorem}\label{T:*-linear}
Let $\cL$ be a linear map as in \eqref{cL}. Define $L\in\BF^{n^2 \times q^2}$ by \eqref{Ldef} and $\BL\in\BF^{nq \times nq}$ by \eqref{BLdef}. Then the following are equivalent:
\begin{itemize}
  \item[(i)] $\cL$ is $*$-linear;
  \item[(ii)] $\BL\in\cH_{nq}$;
  \item[(iii)] $\ov{L}=\fC_n L \fC_q$.
\end{itemize}
Moreover, if one (and hence all) of the above holds, then $\cL\left(\cH_q\right)\subset \cH_n$. Finally, for $\BF=\BC$, if $\cL\left(\cH_q\right)\subset \cH_n$, then (i)--(iii) hold.
\end{theorem}

For $\BF=\BC$, with $*$-linearity replaced by Hermitian-preserving, the equivalence of (i) and (ii) was obtained by Hill in \cite[Theorem 1]{H73}, using the formulation in terms of the entries of $L$, as in \eqref{HermChar2}. The equivalence with (ii) was obtained in \cite[Theorem 2]{PH81}. That for $\BF=\BC$ Hermitian-preserving maps are $*$-linear was noted on Page 260 in \cite{H73}.

\begin{proof}[\bf Proof of Theorem \ref{T:*-linear}]
The equivalence of (ii) and (iii) follows directly from Proposition \ref{P:Herm}. Now we prove the equivalence of (i) and (ii). Assume $\cL$ is $*$-linear. Then by \eqref{BLdef}, for all $1 \le i,j \le q$ we have
\[
\BL_{ij}^*=\cL\left(\cE^{(q)}_{ij}\right)^*=\cL\left(\left(\cE^{(q)}_{ij}\right)^*\right)=\cL\left(\cE^{(q)}_{ji}\right)=\BL_{ji},
\]
so that $\BL\in\cH_{nq}$. Conversely, if $\BL\in\cH_{nq}$, then by \eqref{BL to cL1} it follows for all $V\in\BF^{q\times q}$ that
\[
\cL\left(V^*\right)=\sum_{i,j=1}^q \ov{v}_{ji}\BL_{ij}=\sum_{i,j=1}^q \ov{v}_{ji}\BL_{ji}^*
=\left(\sum_{i,j=1}^q v_{ji}\BL_{ji}\right)^*=\cL(V)^*.
\]
Hence (i), (ii) and (iii) are equivalent. Assuming (i), it follows for all $V\in\cH_q$ that $\cL(V)^*=\cL(V^*)=\cL(V)$. Hence $\cL\left(\cH_q\right)\subset \cH_n$.

Finally, let $\BF=\BC$ and assume $\cL\left(\cH_q\right)\subset \cH_n$. We prove that $\BL\in\cH_{nq}$. Since $\cE_{ij}^{(q)}+\cE_{ji}^{(q)}\in \cH_q$, we obtain that $\BL_{ij}+\BL_{ji}=\cL\left(\cE_{ij}^{(q)}+\cE_{ji}^{(q)}\right)\in \cH_n$. Hence
\[
\BL_{ij}+\BL_{ji}=\BL_{ij}^*+ \BL_{ji}^*.
\]
Similarly, using that $\imag \cE_{ij}^{(q)}- \imag \cE_{ji}^{(q)}\in \cH_q$ it follows that
\[
\BL_{ij}-\BL_{ji}=-\BL_{ij}^*+ \BL_{ji}^*.
\]
Adding the above two identities yields $2\, \BL_{ij}= 2\, \BL_{ji}^*$. Hence $\BL\in\cH_{nq}$.
\end{proof}

By the argument in the last paragraph of the above proof, using only the first identity, for $\BF=\BR$ it still follows that $\BL_{ii}\in \cH_n$ for each $i$; see Theorem 3 in \cite{BG15} where this was noted for a special case. For $\BF=\BR$ conditions (i)-(iii) are not implied by $\cL\left(\cS_q\right)\subset \cS_n$ as seen by the next example.

\begin{example}\label{E:StoSnotenough}
Consider the linear map $\cL: \BR^{2 \times 2} \to \BR^{3\times 3}$ given by
\begin{equation*}
\cL\left(\begin{bmatrix} k_{11} & k_{12} \\ k_{21} & k_{22} \end{bmatrix}\right)=\begin{bmatrix} k_{11}+k_{12} & 0 &0\\ 0 & k_{21}+k_{22} &0 \\ 0& 0&0\end{bmatrix}.
\end{equation*}
Clearly $\cL$ maps $\cS_2$ into $\cS_3$, in fact, $\BR^{2\times 2}$ is mapped into $\cS_3$. However, we have
\begin{equation*}
\begin{aligned}
    \BL_{11}&=\cL\left(\begin{bmatrix} 1 & 0 \\ 0 & 0   \end{bmatrix}\right)=\begin{bmatrix} 1 & 0 &0\\ 0 & 0 &0 \\ 0 & 0 & 0 \end{bmatrix}, \quad \BL_{12}=\cL\left(\begin{bmatrix} 0 & 1\\0& 0\end{bmatrix}\right) =\begin{bmatrix} 1 & 0 &0 \\0& 0 & 0\\0&0&0 \end{bmatrix}, \\ \BL_{21}&=\cL\left(\begin{bmatrix} 0 & 0  \\ 1 & 0 \end{bmatrix}\right)=\begin{bmatrix} 0 &0 & 0 \\ 0 & 1 &0\\0&0&0\end{bmatrix}, \quad \BL_{22}=\cL\left(\begin{bmatrix} 0 & 0 \\ 0 & 1 \end{bmatrix}\right)=\begin{bmatrix} 0 &0 & 0 \\ 0 & 1&0 \\ 0 &0 & 0\end{bmatrix},
\end{aligned}
\end{equation*}
so that
\begin{equation*}
\BL=\begin{bmatrix}
\BL_{11} & \BL_{12} \\ \BL_{21} & \BL_{22}\end{bmatrix}=\begin{bmatrix} 1 & 0 &0 & 1 & 0 &0 \\ 0 & 0 & 0 & 0 &0 &0 \\ 0 & 0 & 0 & 0 &0 &0\\ 0 & 0 & 0 & 0 &0 &0 \\ 0 &1 &0 &0&1&0 \\ 0&0&0&0&0&0 \end{bmatrix}\neq \BL^T.
\end{equation*}
\end{example}

\section{Hill representations of $*$-linear matrix maps}\label{S:HillReps}

In this section we conduct our main analysis of minimal Hill representations of $*$-linear maps and prove our main theorems. For the reader's convenience we recall here that a Hill representation for a $*$-linear map $\cL$ as in \eqref{cL} is a representation of the form
\begin{equation}\label{HillRep}
\cL(V)=\sum_{k,l=1}^m \BH_{kl}\, A_l V A_k^*,\quad V\in\BF^{q \times q},
\end{equation}
for matrices $A_1,\ldots,A_m\in\BF^{n \times q}$. The matrix $\BH=[\BH_{kl}]_{k,l=1}^m\in\BF^{m\times m}$ is called the associated Hill matrix, and we say that a Hill representation \eqref{HillRep} is minimal in case the number $m$ is the smallest among all Hill representations for $\cL$. Note that in a minimal Hill representation, the matrices $A_1,\ldots,A_m$ must be linearly independent, since otherwise the linear dependency can be used to obtain a Hill representation with a smaller number $m$.

Given a linear map  $\cL$ as in \eqref{cL}, define the matricization $L$ by \eqref{Ldef} and the Choi matrix $\BL$ by \eqref{BLdef}. The number $m$ in \eqref{HillRep} for a minimal Hill representation is equal to the rank of the Choi matrix:
\[
m:=\rank \,\BL.
\]
We prove this fact in Corollary \ref{C:minimal} below. Note that the columns of $\BL$ are obtained by vectorizing the block entries $L_{ij}\in \BF^{n \times q}$ of $L$. Therefore, we have that
\begin{equation}\label{mAlt}
m=\dim \tu{span}\{ L_{ij} \colon i=1,\ldots,n, \, j=1,...,q \}.
\end{equation}
Now let $L_1,\ldots,L_m\in\BF^{n \times q}$ so that
\begin{equation}\label{LkCond}
\tu{span}\{ L_{ij} \colon i=1,\ldots,n, \, j=1,...,q \} =\tu{span}\{ L_{k} \colon k=1,\ldots,m \}.
\end{equation}
At this stage we do not assume that the matrices $L_1,\ldots, L_m$ are among the block entries $L_{ij}$ of $L$, but it is always possible to choose them in that way. It follows that there exits scalars $\al^{ij}_k, \be_{ij}^k\in\BF$ for $i=1,\ldots,n$, $j=1,...,q$ and $k=1,\ldots,m$, so that
\begin{equation}\label{LkLij}
L_k =\sum_{i=1}^n\sum_{j=1}^q \be_{ij}^k L_{ij},\quad L_{ij}=\sum_{k=1}^m \al^{ij}_k L_k.
\end{equation}
Note that the matrices $L_1,\ldots, L_m$ form a linearly independent set.
Hence the scalars $\al_k^{ij}$ are uniquely determined. This need not be the case for the $\be_{ij}^k$. Set
\begin{equation}\label{AkBk}
A_k=\left[\ov{\al}_{k}^{ij}\right]\in \BF^{n \times q} \quad \text{and} \quad B_k=\left[\be_{ij}^k\right]\in\BF^{n \times q} \quad \text{for} \quad k=1,\ldots,m.
\end{equation}
Then we have
\begin{equation}\label{LLk}
L=\sum_{k=1}^m \ov{A}_k \otimes L_k \ands L_k=\left(\OneVec_n \otimes I_n\right)^*\left(\left(B_k \otimes \BBone_{n \times q}\right) \circ L\right)\left(\OneVec_q\otimes I_q\right).
\end{equation}
Next we define what we call the {\em Hill matrix} associated with $\cL$ and the matrices $L_1,\ldots,L_m$:
\begin{equation}\label{Hill}
\BH=\BH\left(\cL;L_1,\ldots,L_m\right):=\left[\OneVec_n^*\left(B_k \circ \ov{L}_l\right)\OneVec_q\right]_{k,l=1}^m\in\BF^{m \times m}.
\end{equation}
In case $\cL$ is $*$-linear, $\cL$ admits a minimal Hill representation with $A_1,\ldots,A_m$ and $\BH$ as constructed above, as follows from the next theorem, which also provides representations for $L$ and $\BL$.

\begin{theorem}\label{T:HillRep}
Assume $\cL$ as in \eqref{cL} is $*$-linear. Define $L$ as in \eqref{Ldef} and $\BL$ as in \eqref{BLdef} and let $m=\rank \, \BL$. Choose $L_1,\ldots,L_m\in\BF^{n \times q}$ so that \eqref{LkCond} holds and define $A_1,\ldots,A_m$ and $\BH$ as above. Then
\begin{equation}\label{cLLBL2}
\begin{aligned}
\cL(V)&=\sum_{k,l=1}^m \BH_{kl} \, A_l V A_k^*,\quad
L=\sum_{k,l=1}^m \BH_{kl}\, \ov{A}_k\otimes A_l,\\
\BL&=\sum_{k,l=1}^m \BH_{kl}\, \vect_{n\times q}\left(A_l\right)\vect_{n \times q}\left(\ov{A}_k\right)^T=\widehat{A}^*\BH^T\widehat{A},
\end{aligned}
\end{equation}
with $\widehat{A}^*:=\begin{bmatrix} \vect_{n \times q}\left(A_1\right) & \hdots & \vect_{n \times q}\left(A_m\right)  \end{bmatrix}\in \BF^{nq \times m}$. Furthermore $\BH$ is in $\cH_m$ and invertible. Alternatively, these identities can be written as
\begin{align}
\cL(V)&= \left(\OneVec_m\otimes I_{n}\right)^* \left(\left(\BH \otimes \BBone_{n}\right) \circ \left[
A_l V A_k^*\right]_{k,l=1}^m\right) \left(\OneVec_m\otimes I_{n}\right),\notag\\
L&=\left(\OneVec_m\otimes I_{n^2}\right)^* \left(\left(\BH \otimes \BBone_{n^2 \times q^2}\right) \circ \left[\ov{A}_k\otimes A_l\right]_{k,l=1}^m\right) \left(\OneVec_m\otimes I_{q^2}\right),  \label{cLLBL1}\\
\BL&=\left(\OneVec_m\otimes I_{nq}\right)^* \left(\left(\BH \otimes \BBone_{nq}\right) \circ \left[\vect_{n\times q}\left(A_l\right)\vect_{n\times q}\left(\ov{A}_k\right)^T\right]_{k,l=1}^m\right) \left(\OneVec_m\otimes I_{nq}\right).\notag
\end{align}
\end{theorem}

The proof will be given later in this section after some auxiliary results.  The fact that $\BH$ is in $\cH_m$ follows directly from \cite[Theorem 4]{PH81}, however, we will also provide a proof for this fact.

We start with a result that also holds when $\cL$ is not $*$-linear.

\begin{lemma}
For $L_1,\ldots,L_m$ as in \eqref{LkCond} and $A_1,\ldots,A_m$ and $B_1,\ldots,B_m$ defined by \eqref{AkBk} and \eqref{LkLij} we have
\begin{equation}\label{AkBkRel}
\OneVec_n^*\left(B_k \circ \ov{A}_k\right)\OneVec_q=1 \quad \text{and} \quad \OneVec_n^*\left(B_k \circ \ov{A}_l\right)\OneVec_q=0 \quad \text{for} \quad l \neq k.
\end{equation}
\end{lemma}

\begin{proof}[\bf Proof]
Using both identities in \eqref{LLk} we find that
\begin{align*}
L_k & = \left(\OneVec_n \otimes I_n\right)^*\left(\left(B_k \otimes \BBone_{n \times q}\right) \circ L\right)\left(\OneVec_q\otimes I_q\right) \\
&= \sum_{l=1}^m \left(\OneVec_n \otimes I_n\right)^*\left(\left(B_k \otimes \BBone_{n \times q}\right) \circ (\ov{A}_l \otimes L_l)\right)\left(\OneVec_q\otimes I_q\right)\\
&= \sum_{l=1}^m \left(\OneVec_n \otimes I_n\right)^*\left(\left(B_k \circ \ov{A}_l\right) \otimes \left(\BBone_{n \times q} \circ L_l\right)\right)\left(\OneVec_q\otimes I_q\right) \ \ \text{(using \eqref{HadTens})}\\
&= \sum_{l=1}^m \left(\OneVec_n \otimes I_n\right)^*\left(\left(B_k \circ \ov{A}_l\right) \otimes  L_l\right)\left(\OneVec_q\otimes I_q\right)
= \sum_{l=1}^m \OneVec_n^*\left(B_k \circ \ov{A}_l\right)\OneVec_q \,  L_l.
\end{align*}
The identities in \eqref{AkBkRel} now follow by the linear independence of $L_1,\ldots,L_m$.
\end{proof}

\begin{proposition}\label{P:*linearChar}
Let $\cL$ as in \eqref{cL} be linear and let $L_1,\ldots,L_m$ be as in \eqref{LkCond}. Define $A_1,\ldots,A_m$ and $B_1,\ldots,B_m$ by \eqref{AkBk} and \eqref{LkLij}. Then the following are equivalent:
\begin{itemize}
\item[(i)] $\cL$ is $*$-linear;

\item[(ii)] $L=\sum_{k=1}^m \ov{L}_k \otimes A_k$;

\item[(iii)] $L_k= \sum_{l=1}^m \OneVec_n^* \left(B_k \circ \ov{L}_l\right) \OneVec_q\, A_l$ for $k=1,\ldots,m.$

\end{itemize}
Furthermore, if one of the above holds, and hence all, then
\begin{equation}\label{SelfAdjId}
\OneVec_n^*\left(B_k\circ \overline{L}_l\right)\OneVec_q=\OneVec_n^*\left(\ov{B}_l\circ L_k\right)\OneVec_q \quad \text{for} \quad k,l=1,\ldots,m.
\end{equation}
\end{proposition}

\begin{proof}[\bf Proof]
Note that since $L$ is as in \eqref{LLk}, we have
\[
\fC_n L \fC_q=\sum_{k=1}^m \fC_n\left( \ov{A}_k \otimes L_k\right) \fC_q= \sum_{k=1}^m L_k \otimes \ov{A}_k.
\]
Hence $\ov{\fC_n L \fC_q}=\sum_{k=1}^m \ov{L}_k \otimes A_k$. The equivalence of (i) and (ii) now follows directly from the equivalence of (i) and (iii) in Theorem \ref{T:*-linear}.

Next we show that (ii) implies (iii). Assume (ii). Hence
\[
\sum_{l=1}^m \ov{A}_l \otimes L_l=L =\sum_{l=1}^m \ov{L}_l \otimes A_l.
\]
Using \eqref{AkBkRel}, for $k=1,\ldots,m$ this implies that
\begin{align*}
L_k&= \sum_{l=1}^m \OneVec_n^*\left(B_k \circ \ov{A}_l\right)\OneVec_q\, L_l
= \sum_{l=1}^m  \left(\OneVec_n\otimes I_n\right)^* \left(\left(B_k\otimes \BBone_{n\times q}\right) \circ \left(\ov{A}_l \otimes L_l\right)\right) \left(\OneVec_q\otimes I_q\right)\\
&= \sum_{l=1}^m  \left(\OneVec_n\otimes I_n\right)^* \left(\left(B_k\otimes \BBone_{n \times q}\right) \circ \left(\ov{L}_l \otimes A_l\right)\right) \left(\OneVec_q\otimes I_q\right)
=\sum_{l=1}^m \OneVec_n^*\left(B_k \circ \ov{L}_l\right)\OneVec_q\, A_l.
\end{align*}

Now assume that (iii) holds.  Using \eqref{AkBkRel} we find that
\begin{align*}
& \OneVec_n^*\left(B_k\circ \overline{L}_l\right)\OneVec_q
 =\sum_{r=1}^m \OneVec_n^*\left(\ov{B}_l\circ A_r\right)\OneVec_q \,  \OneVec_n^*\left(B_k\circ \overline{L}_r\right)\OneVec_q\\
  &\qquad\qquad= \OneVec_n^*\left(\ov{B}_l\circ \left(\sum_{r=1}^m \OneVec_n^*\left(B_k\circ \overline{L}_r\right)\OneVec_q\, A_r \right)\right)\OneVec_q
  = \OneVec_n^*\left(\ov{B}_l\circ L_k\right)\OneVec_q.
\end{align*}
Hence (iii) implies \eqref{SelfAdjId}. To complete the proof we show that (ii) holds, still assuming (iii), which means also \eqref{SelfAdjId} holds. Using these we find that
\begin{align*}
L & = \sum_{k=1}^m \ov{A}_k \otimes L_k
= \sum_{k,l=1}^m \OneVec_n^*\left(B_k\circ \overline{L}_l\right)\OneVec_q\, \ov{A}_k \otimes  A_l\\
&= \sum_{k,l=1}^m \OneVec_n^*\left(\ov{B}_l\circ L_k\right)\OneVec_q\, \ov{A}_k \otimes  A_l
= \sum_{l=1}^m \ov{L}_l \otimes A_l.\qedhere
\end{align*}
\end{proof}

As a direct consequence of Proposition \ref{P:*linearChar}, either (ii) or (iii), we obtain the following observation.

\begin{corollary}\label{C:spanA1Am}
Assume $\cL$ in \eqref{cL} is $*$-linear. Choose $L_1,\ldots,L_m$ so that \eqref{LkCond} holds and define $A_1,\ldots,A_m$ and $B_1,\ldots,B_m$ by \eqref{AkBk} and \eqref{LkLij}.
Then
\[
\tu{span}\{A_1,\ldots, A_m\}=\tu{span}\{L_{ij}\colon i=1,\ldots,n,\, j=1,\ldots,q\}= \tu{span}\{L_1,\ldots, L_m\}.
\]
\end{corollary}
Before we prove Theorem \ref{T:HillRep} we first prove a useful Lemma.
\begin{lemma}\label{L:MatTrans}Assume $\cL$ in \eqref{cL} is $*$-linear and define $L$ by \eqref{Ldef}. Let $K_1,\ldots, K_m \in\BF^{n \times q}$ so that
\[
\tu{span}\{K_1,\ldots, K_m\}=\tu{span}\{L_{ij}\colon i=1,\ldots,n,\, j=1,\ldots,q\}.
\]
Set
\begin{equation}\label{MatTrans}
\wtilK=\mat{K_1\\\vdots\\ K_m}\in\BF^{mn \times q}  \ands
\whatK=\mat{\widehat{K}_1 \\ \vdots \\ \widehat{K}_m}\in\BF^{m \times nq}
\end{equation}
where $\whatK_l:=\vect_{n \times q}\left(\ov{K}_l\right)^{T}$ for $l=1,\ldots,m$. Then $\rank \whatK=m$, hence $\whatK$ has full row rank and we have\begin{equation}\label{KtilKhatRel}\left(I_m\otimes x\right)^T \wtilK = \ov{\whatK} \left(I_q \otimes x \right),\quad x\in\BF^n.\end{equation}
\end{lemma}

\begin{proof}[\bf Proof]
Using \eqref{VectForm1} we obtain \eqref{KtilKhatRel} from the following computation:\begin{align*}\ov{\widehat{K}}\left(I_q \otimes x\right)&=\mat{\vect_{n \times q}\left(K_1\right)^T\left(I_q \otimes x\right)\\\vdots \\ \vect_{n \times q}\left(K_m\right)^T\left(I_q \otimes x\right)}=\begin{bmatrix}x^TK_1 \\ \vdots \\ x^TK_m \end{bmatrix}=\left(I_m \otimes x\right)^T\widetilde{K}.\end{align*}
To see that $\widehat{K}$ has full row rank, let $v\perp \im \whatK$. Then for all $u \in \BF^{nq}$ we have
\begin{equation*}
\begin{aligned}
0 &= v^*\widehat{K}u
= v^*\begin{bmatrix} \vect_{n \times q}\left(\ov{K}_1 \right)^Tu \\ \vdots \\ \vect_{n \times q}\left(\ov{K}_m\right)^Tu \end{bmatrix} \\ &
= \sum_{l=1}^m \ov{v}_l \vect_{n \times q}\left(\ov{K}_l\right)^Tu
= \vect_{n \times q}\left(\sum_{l=1}^m \ov{v}_l \ov{K}_l \right)^Tu.
\end{aligned}
\end{equation*}This implies that $\vect_{n \times q}\left(\sum_{l=1}^m \ov{v}_l \ov{K}_l \right)=0$, which is true if and only if $\sum_{l=1}^m \ov{v}_l \ov{K}_l=0$. Since $K_1,\ldots,K_m$ are linearly independent, so are $\ov{K}_1,\ldots,\ov{K}_m$, and hence $v_1=\cdots = v_m=0$. Thus $\left(\im \whatK\right)^\perp=\{0\}$, proving that $\im \whatK=\BF^m$ and $\rank \whatK=m$.
\end{proof}

\begin{proof}[\bf Proof of Theorem \ref{T:HillRep}]
The equivalence of the formulas in \eqref{cLLBL2} and the corresponding formulas in \eqref{cLLBL1} follows directly from the identity \eqref{SumCirc}. To see that $L$ is given by \eqref{cLLBL2}, simply insert the formula for $L_k$ in item (iii) of Proposition \ref{P:*linearChar} into the formula for $L$ in \eqref{LLk}.

To see that $\BL$ is as in \eqref{cLLBL2}, let $\La=\La_{(n,q)}^{(n,q)}$ and use \eqref{InvolId1} to obtain
\begin{align*}
\BL & =\La(L)=\sum_{k,l=1}^m \BH_{kl}\, \La\left(\ov{A}_k\otimes A_l\right)
= \sum_{k,l=1}^m \BH_{kl}\, \vect(A_l) \vect\left(\ov{A}_k\right)^T,
\end{align*}
which is equivalent to $\BL=\widehat{A}^*\BH^T\widehat{A},$ with $\widehat{A}$ defined as in the theorem. For the formula for $\cL$, combine \eqref{Linv} with the formula for $L$ in \eqref{cLLBL2} together with the identity:
\begin{align*}
\cL(V) & = \vect^{-1}\left(\sum_{k,l=1}^m \BH_{kl}\, \left(\ov{A}_k\otimes A_l\right) \vect(V)\right)
= \vect^{-1}\left(\sum_{k,l=1}^m \BH_{kl}\,  \vect\left(A_l V A_k^*\right)\right)\\
&= \sum_{k,l=1}^m \BH_{kl} \vect^{-1}\left( \vect(A_l V A_k^*)\right)
=\sum_{k,l=1}^m \BH_{kl}\, A_l V A_k^*.
\end{align*}
That $\BH$ is in $\cH_m$ is just are reformulation of \eqref{SelfAdjId}. Hence it remains to show that $\BH$ is invertible, or equivalently, $\rank \, \BH=m$. We use that $\BL=\what{A}^*\BH^T\what{A}$ and $\rank \what{A}= m$ to conclude that $m=\rank \BL = \rank \BH^T= \rank \BH.$
\end{proof}

We can now prove our claim regarding the number of matrices $A_k$ in a minimal Hill representation.

\begin{corollary}\label{C:minimal}
The minimal number of matrices $A_k$ appearing in a Hill representation is equal to $\rank \BL$.
\end{corollary}
\begin{proof} [\bf Proof] This follows directly from $\rank \BL = \rank \BH$, since the minimum number of matrices $A_k$ in a Hill representation is equal to the size of $\BH$, and we know $\BH$ is invertible. \end{proof}


Since $\BH$ in Theorem \ref{T:HillRep} is Hermitian we have $\ov{\BH}=\ov{\BH}^*=\BH^T$. Together with the formula for $\BL$ in \eqref{cLLBL2} and the fact that $\BH$ is invertible, yields the following corollary.

\begin{corollary}\label{C:BLform}
Assume $\cL$ as in \eqref{cL} is $*$-linear. Define $\BL$ as in \eqref{BLdef} and set $m=\rank \, \BL$.
Choose $L_1,\ldots,L_m\in\BF^{n \times q}$ so that \eqref{LkCond} holds and define $\whatA$ as in Theorem \ref{T:HillRep}. Then $\BL=\widehat{A}^*\ov{\BH}\widehat{A}$. Moreover, $\whatA$ has full row rank and $\kr \what A=\kr \BL$.
\end{corollary}

In particular, the matrices $A_1,\ldots,A_m$ in the minimal Hill representation of Theorem \ref{T:HillRep} also form a basis for the span of the block matrices $L_{ij}$ in $L$. A closer inspection of item (ii) Proposition \ref{P:*linearChar} in fact tells us that in case $\cL$ is $*$-linear, then a minimal Hill representation for $\cL$ can be formed with any collection $A_1,\ldots,A_m$ in $\BF^{n \times q}$ that forms a basis for the span of the matrices $L_{ij}$.

\begin{theorem}\label{T:HillRepChar}
Let $\cL$ as in \eqref{cL} be $*$-linear. Define $L$ as in \eqref{Ldef}, $\BL$ as in \eqref{BLdef} and let $m=\rank \, \BL$. For all $A_1,\ldots,A_m \in\BF^{n \times q}$ so that
\begin{equation}\label{SpanCond}
\tu{span}\{A_1,\ldots, A_m\}=\tu{span}\{L_{ij}\colon i=1,\ldots,n,\, j=1,\ldots,q\},
\end{equation}
$\cL$ admits a minimal Hill representation \eqref{HillRep} for some $\BH\in\cH_m$. Conversely, if $\cL$ admits a minimal Hill representation \eqref{HillRep}, then the matrices $A_1,\ldots,A_m$ in \eqref{HillRep} satisfy \eqref{SpanCond}. Furthermore, the Hill matrix $\BH$ associated with $\cL$ and the matrices $L_1,\ldots, L_m$ is uniquely determined by
\begin{equation}\label{HillMatRel}\begin{aligned}
\mat{L_1\\ \vdots \\ L_m} &=\left(\BH \otimes I_n\right) \mat{A_1\\ \vdots \\ A_m},\quad \mbox{where} \quad L_k=\left[\ov{\la}^{ij}_k\right]\in \BF^{n \times q},\\ \mbox{with} \quad L_{ij}&=\sum_{k=1}^m\la^{ij}_k A_k \quad \text{for} \quad i=1,...,n \quad \text{and} \quad j=1,...,q.\end{aligned}
\end{equation}
\end{theorem}

\begin{proof}[\bf Proof]
Let $A_1,\ldots,A_m\in\BF^{n\times q}$ satisfy \eqref{SpanCond}. Then there exist scalars $\la^{ij}_k \in \BF$ for $i=1,\ldots,n$, $j=1,\ldots,q$ and $k=1,\ldots,m$ so that
\[
L_{ij}=\sum_{k=1}^m \la^{ij}_k A_k,\quad i=1,\ldots,n, \quad j=1,\ldots,q.
\]
Then for $L_k=\left[\ov{\la}^{ij}_k\right]\in \BF^{n \times q}$, $k=1,\ldots,m$, we have $L=\sum_{k=1}^m \ov{L}_k \otimes A_k$. Using this representation for $L$, instead of that in \eqref{LLk}, it follows from Proposition \ref{P:*linearChar} that
\[
L=\sum_{k=1}^m \ov{A}_k \otimes L_k.
\]
Hence $L$ is also as in the first identity of \eqref{LLk}. By Theorem \ref{T:HillRep} $\cL$ admits a minimal Hill representation \eqref{HillRep} with the selected $A_1,\ldots,A_m$ and with $\BH$ as in \eqref{Hill} Hermitian.

Conversely, assume $\cL$ admits a minimal Hill representation \eqref{HillRep}. By the same argument as in the proof of Theorem \ref{T:HillRep}, we have
\begin{align*}
L & = \sum_{k,l=1}^m \BH_{kl}\, \ov{A}_k \otimes A_l =\sum_{l=1}^m \left(\sum_{k=1}^m \BH_{kl} \ov{A}_k\right)\otimes A_l.
\end{align*}
This implies that $L_{ij}=\sum_{l=1}^m \left(\sum_{k=1}^m \BH_{kl} \al^{ij}_k\right) A_l$, for all $i,j$, with $\al^{ij}_k$ as in \eqref{AkBk}.  Hence the span of the matrices $L_{ij}$ is contained in the span of $A_1,\ldots,A_m$. Now the equality \eqref{SpanCond} follows by \eqref{mAlt} and a dimension argument.

Finally, to see that \eqref{HillMatRel} holds, note that the above formula for $L$ shows that $L$ admits a representation as in \eqref{LLk} with $L_l =\sum_{k=1}^m \ov{\BH}_{kl} A_k$, $l=1,\ldots,m$, which translates to \eqref{HillMatRel}. That the matrices $L_1,\ldots,L_m$ are obtained from $A_1,\ldots,A_m$ follows from the analysis in the first part of the proof.
\end{proof}

\begin{proof}[\bf Proof of Theorem \ref{T:main1}]
That for any choice of matrices $L_1,\ldots,L_m$ satisfying \eqref{SpanL1Lm}, the right hand side of \eqref{HillRepIntro} with $A_1,\ldots,A_m$ as in \eqref{AkBk-intro} and $\BH$ as in \eqref{Hill-intro} provides a minimal Hill representation follows from Theorem \ref{T:HillRep} together with Corollary \ref{C:minimal}. By Corollary \ref{C:spanA1Am} it follows that the matrices $A_1,\ldots,A_m$ satisfy \eqref{A1Amcon}. Conversely, we obtain from Theorem \ref{T:HillRepChar} that any $A_1,\ldots,A_m$ satisfying \eqref{A1Amcon} appear as the matrices in a minimal Hill representation. Finally, for any minimal Hill representation \eqref{HillRep} of $\cL$, using that $L$ is as in \eqref{cLLBL2} it follows that the block entries $L_{ij}$ of $L$ are all in the span of $A_1,\ldots,A_m$, which together with the fact that $m$ is equal to \eqref{mAlt}, by Corollary \ref{C:minimal}, implies that \eqref{A1Amcon} holds.
\end{proof}

The identity \eqref{HillMatRel} shows how the matrices $L_1,\ldots,L_m$ and $A_1,\ldots,A_m$ in the representation \eqref{LLk} are related when $\cL$ is $*$-linear, and determine each other uniquely, since $\BH$ is invertible. In fact, for two selections $L_1,\ldots,L_m$ and $L'_1,\ldots,L_m'$ of matrices in $\BF^{n \times q}$ so that their span coincides with the span of the matrices $L_{ij}$ such a relation exist as well as between the associated Hill matrices.

Next we explain how the matrices associated with two sets of linear independent matrices $L_1,\ldots,L_m$ and $L_1^\prime,\ldots,L_m^\prime$ so that their spans correspond with the span of the matrices $L_{ij}$ that constitute $L$ relate.

\begin{theorem}\label{T:LkLk'Rel}
Assume $\cL$ as in \eqref{cL} is $*$-linear. Define $L$ as in \eqref{Ldef}, $\BL$ as in \eqref{BLdef} and set $m=\rank \, \BL$. Let $L_1,\ldots,L_m$ and $L_1^\prime,\ldots,L_m^\prime$ be  in $\BF^{n \times q}$ so that
\begin{equation*}
\tu{span}\{L_1,\ldots,L_m\}=\tu{span}\{L_{ij} \colon i=1,\ldots,n, \,j=1,..,q\}=\tu{span}\{L_1^\prime,\ldots,L_m^\prime\}.
\end{equation*}
For $k=1,\ldots,m$ define $A_k$ and $B_k$ as in \eqref{AkBk} and $\BH$ as in \eqref{Hill}. Analogously, define $A_k^\prime$ and $B_k^\prime$ and $\BH^\prime$ for $L_1^\prime,\ldots,L_m^\prime$. Also, define $\whatL$, $\whatL^\prime$, $\widetilde{L}$, $\widetilde{L}^\prime$, $\whatA$, $\whatA^\prime$, $\widetilde{A}$ and $\widetilde{A}^\prime$ in analogy to \eqref{MatTrans} for $L_1,\ldots,L_m$,  $L_1^\prime,\ldots,L_m^\prime$, $A_1,\ldots,A_m$ and $A_1^\prime,\ldots,A_m^\prime$, respectively. Define
\begin{equation}\label{PhiXi}
\Phi:=\left[\OneVec_n^*\left(B_k\circ \ov{A}_l^{\prime}\right)\OneVec_q\right]_{k,l=1}^m
\ands
\Xi :=\left[\OneVec_n^*\left(B_k\circ \ov{L}_l^{\prime}\right)\OneVec_q\right]_{k,l=1}^m.
\end{equation}
Then
\begin{equation}\label{whatRels}
\begin{aligned}
\whatL=\ov{\BH}\whatA,\quad  \whatL^\prime=\ov{\BH}^\prime\whatA^\prime,\quad
\whatL=\ov{\Phi}\whatL^\prime,\quad  \ov{\Phi}^* \whatA=\whatA^\prime,\\
\whatL=\ov{\Xi}\whatA^\prime,\quad \whatL^\prime=\ov{\Xi}^* \whatA
\ands \BH=\Phi \BH'\Phi^*.
\end{aligned}
\end{equation}
Moreover, $\BH=\Phi\Xi^*$, $\Phi\BH^\prime=\Xi$, and $\Phi$ and $\Xi$ are invertible with inverses
\[
\Phi^{-1}=\left[\OneVec_n^*\left(B_k^{\prime}\circ \ov{A}_l\right)\OneVec_q\right]_{k,l=1}^m \ands
\Xi^*= \left[\OneVec_n^*\left(B_k^\prime\circ \ov{L}_l\right)\OneVec_q\right]_{k,l=1}^m.
\]
Furthermore, $\kr \whatL=\kr\whatL^\prime=\kr\whatA=\kr\whatA^\prime=\kr\BL$.
\end{theorem}

\begin{proof}[\bf Proof]
For now, set
\[
\Phi^\prime:=\left[\OneVec_n^*\left(B_k^{\prime}\circ \ov{A}_l\right)\OneVec_q\right]_{k,l=1}^m \ands
\Xi^\prime:=\left[\OneVec_n^*\left(B_k^\prime\circ \ov{L}_l\right)\OneVec_q\right]_{k,l=1}^m.
\]
Later in the proof we show that $\Phi^\prime=\Phi^{-1}$ and $\Xi^\prime=\Xi^*$.

Since $L=\sum_{l=1}^m \ov{A}_l^\prime \otimes L_l^\prime$ we obtain from \eqref{LLk} that
\begin{align*}
L_k
&=\left(\OneVec_n \otimes I_n\right)^*\left(\left(B_k \otimes \BBone_{n \times q}\right) \circ L\right)\left(\OneVec_q\otimes I_q\right)\\
&= \sum_{l=1}^m \left(\OneVec_n \otimes I_n\right)^*\left(\left(B_k \otimes \BBone_{n \times q}\right) \circ (\ov{A}_l^\prime \otimes L_l^\prime)\right)\left(\OneVec_q\otimes I_q\right)\\
&= \sum_{l=1}^m \left(\OneVec_n \otimes I_n\right)^*
\left((B_k \circ \ov{A}_l^\prime ) \otimes (\BBone_{n \times q} \circ L_l^\prime)
\right)\left(\OneVec_q\otimes I_q\right)\\
&= \sum_{l=1}^m \left(\OneVec_n \otimes I_n\right)^*
\left((B_k \circ \ov{A}_l^\prime ) \otimes  L_l^\prime
\right)\left(\OneVec_q\otimes I_q\right)
= \sum_{l=1}^m \OneVec_n^* (B_k \circ \ov{A}_l^\prime )\OneVec_q\, L_l^\prime.
\end{align*}
This proves that $\wtilL=(\Phi \otimes I_n) \wtilL^\prime$. According to Proposition \ref{P:*linearChar}, also $L=\sum_{l=1}^l \ov{L}_l^\prime \otimes A_l^\prime$. Repeating the above computation with this formula for $L$ yields $\wtilL=(\Xi \otimes I_n) \wtilA^\prime$. By Theorem \ref{T:HillRepChar} we also have $\wtilL=(\BH \otimes I_n) \wtilA$. Hence we have
\[
\wtilL=(\Phi \otimes I_n) \wtilL^\prime,\quad \wtilL=(\Xi \otimes I_n) \wtilA^\prime,\quad \wtilL=(\BH \otimes I_n) \wtilA.
\]
Interchanging the roles of $L_1,\ldots,L_m$ and $L_1^\prime,\ldots,L_m^\prime$ yields
\[
\wtilL^\prime=(\Phi^\prime \otimes I_n) \wtilL,\quad \wtilL^\prime=(\Xi^\prime \otimes I_n) \wtilA,\quad \wtilL^\prime=(\BH^\prime \otimes I_n) \wtilA^\prime.
\]
From $\wtilL=(\BH \otimes I_n) \wtilA$, using \eqref{KtilKhatRel}, for all $x\in\BF^n$ and $z\in\BF^q$ we find that
\begin{align*}
\BH \ov{\widehat{A}}\left(z \otimes x\right)
&= \BH\left(I_m \otimes x \right)^T\widetilde{A}z
= \left(\BH \otimes x^T \right)\widetilde{A}z =\left(I_m \otimes x^T\right)\left(\BH \otimes I_n \right)\widetilde{A}z\\
&=\left(I_m \otimes x\right)^T\widetilde{L}z
= \ov{\widehat{L}}\left(z \otimes x\right).
\end{align*}
Since all vectors in $\BF^{nq}=\BF^{n} \otimes \BF^q$ can be written as sums of pure tensors, it follows that $\BH \ov{\widehat{A}}=\ov{\widehat{L}}$, or equivalently, $\whatL=\ov{\BH}\whatA$. A similar argument applies to the other identities obtained above, resulting in
\begin{equation}\label{whatIds}
\whatL=\ov{\Phi} \whatL^\prime,\quad \whatL=\ov{\Xi} \whatA^\prime,\quad \whatL=\ov{\BH} \whatA,\quad
\whatL^\prime=\ov{\Phi}^\prime \whatL,\quad \whatL^\prime=\ov{\Xi}^\prime \whatA,\quad \whatL^\prime=\ov{\BH}^\prime \whatA^\prime.
\end{equation}
By Lemma \ref{L:MatTrans}, $\whatL$, $\whatA$, $\whatL^\prime$ and $\whatA^\prime$ all have full row rank. Hence $\Phi$, $\Xi$, $\BH$, $\Phi^\prime$, $\Xi^\prime$ and $\BH^\prime$ are all invertible; for $\BH$ and $\BH^\prime$ this already follows from Theorem \ref{T:HillRep}. It then also follows from the above identities that
$\kr \whatL=\kr\whatL^\prime=\kr\whatA=\kr\whatA^\prime$. The formula $\BL=\whatA^* \BH \whatA$ from Theorem \ref{T:HillRep} together with the fact that $\rank \BL=\rank \BH=\rank \whatA$ implies that $\kr \whatA=\kr\BL$. Furthermore, the first and fourth identity show that $\Phi^\prime=\Phi^{-1}$. In particular, we have now proved the first three identities as well as the fifth in \eqref{whatRels} and found the formula for $\Phi^{-1}$.

Using \eqref{whatIds} we also obtain that $\ov{\BH}\whatA=\whatL =\ov{\Phi}\whatL^\prime=\ov{\Phi}\ov{\Xi}^\prime \whatA$. Hence $\BH=\Phi\Xi^\prime$. Likewise we have $\BH^\prime=\Phi^\prime\Xi=\Phi^{-1}\Xi$, so that $\Xi=\Phi\BH^\prime$ and $\Xi^*=\BH^\prime \Phi^*$. Write $\phi_{kl}=\OneVec_n^*\left(B_k\circ \ov{A}_l^{\prime}\right)\OneVec_q$ so that $\Phi=[\phi_{kl}]_{k,l=1}^m$ and $L_l=\sum_{r=1}^m \phi_{lr} L_r^\prime$. Then
\begin{align*}
\Xi^\prime
&=\left[\OneVec_n^*\left(B_k^\prime\circ \ov{L}_l\right)\OneVec_q\right]_{k,l=1}^m
= \left[\OneVec_n^*\left(B_k^\prime\circ \sum_{r=1}^m \ov{\phi}_{lr} \ov{L}_r^\prime\right)\OneVec_q\right]_{k,l=1}^m\\
&= \left[\sum_{r=1}^m  \OneVec_n^*\left(B_k^\prime\circ  \ov{L}_r^\prime\right)\OneVec_q\, \ov{\phi}_{lr}\right]_{k,l=1}^m
= \left[\OneVec_n^*\left(B_k^\prime\circ  \ov{L}_r^\prime\right)\OneVec_q\right]_{k,r=1}^m
\left[\, \ov{\phi}_{lr} \right]_{r,l=1}^m\\
&=\BH^\prime \Phi^* =\Xi^*.
\end{align*}
We have now also proved the sixth identity in \eqref{whatRels}, the formula for $\Xi^*$ as well as the identities $\BH=\Phi\Xi^*$ and $\Xi=\Phi\BH^\prime$.

Using the above identities we find that
\[
\BH=\Phi \Xi^*=\Phi (\Phi \BH^\prime)^*=\Phi \BH^\prime \Phi^*,
\]
and
\[
\ov{\BH}\whatA=\whatL=\ov{\Phi}\whatL^\prime=\ov{\Phi}\ov{\BH}^\prime \whatA^\prime =\ov{\BH} (\ov{\Phi}^*)^{-1}\whatA^\prime.
\]
Since $\BH$ is invertible, we find that $\whatA=(\ov{\Phi}^*)^{-1}\whatA^\prime$, and hence $\ov{\Phi}^*\whatA=\whatA^\prime$. This completes the proof.
\end{proof}

We shall now prove the second main result.

\begin{proof}[\bf Proof of Theorem \ref{T:main2}]
Define $\Phi$ as in \eqref{PhiXi}. Then $\Phi$ is invertible, by Theorem \ref{T:LkLk'Rel}, the last identity in \eqref{LL'AA'forms} corresponds to the last identity in \eqref{whatRels}, while the remaining two identities in \eqref{LL'AA'forms} are obtained from the last two identities on the first line of \eqref{whatRels} using Lemma \ref{L:MatTrans}; see the proof of Theorem \ref{T:LkLk'Rel} for further details.
\end{proof}

The characteristic feature of the matrix $\whatA$ in Theorem \ref{T:HillRep} is $\kr \whatA= \kr \BL$.

\begin{proposition}\label{P:whatAchar}
Assume $\cL$ as in \eqref{cL} is $*$-linear and let $m$ be the rank of the Choi matrix $\BL$ associated with $\cL$. Let $\whatA\in\BF^{m \times nq}$ with $\kr \whatA=\kr \BL$. Then $\whatA$ has full row rank, so that $\whatA\whatA^*$ is invertible, and we have $\BL=\whatA^* \BH^T \whatA$ with
\[
\BH^T=(\whatA\whatA^*)^{-1} \whatA \BL \whatA^* (\whatA\whatA^*)^{-1}.
\]
In particular, $\cL$ admits a minimal Hill representation \eqref{HillRep} with $A_k=\vect^{-1}_{n\times q} (\what{a}_k^T)$, $k=1,\ldots,m$, where $\what{a}_k$ is the $k$-th row of $\whatA$.
\end{proposition}

\begin{proof}[\bf Proof]
Since $\kr \whatA= \kr \BL$, we have $\rank \whatA=\rank \BL=m$, so that $\whatA$ has full row rank. This implies that $\whatA\whatA^*$ is invertible and that the projection on $\im \whatA^*=\kr\whatA^\perp$ is given by $P_{\im \whatA}=\whatA^*(\whatA\whatA^*)^{-1}\whatA$. Since $\BL$ is selfadjoint, $\im \BL=\kr \BL^\perp=\im\whatA^*$ and thus
\[
\BL=P_{\im \whatA} \BL P_{\im \whatA} = \whatA^*(\whatA\whatA^*)^{-1}\whatA \BL \whatA^*(\whatA\whatA^*)^{-1}\whatA = \whatA^* \BH^T \whatA. \qedhere
\]
\end{proof}

\begin{remark}
On various places in this section, e.g., \eqref{Hill} and \eqref{PhiXi}, we encountered matrices of the form
\[
\left[\OneVec_n^*\left(M_k\circ \ov{N}_l\right)\OneVec_q\right]_{k,l=1}^m,
\]
for matrices $M_k,N_k\in\BF^{n\times q}$, $k=1,\ldots,m$. Note that, via \eqref{TraceInn}, we have
\[
\OneVec_n^*\left(M_k\circ \ov{N}_l\right)\OneVec_q= \trace(M_k N_l^*)=\inn{M_k}{N_l}_{\BF^{n \times q}},
\]
with $\inn{\ }{\ }_{\BF^{n \times q}}$ indicating the trace inner product on $\BF^{n \times q}$. Hence the matrix can also be written as
\[
\left[\inn{M_k}{N_l}_{\BF^{n \times q}}\right]_{k,l=1}^m.
\]
\end{remark}

\section{Examples and special cases}\label{S:Example}

When $\cL$ in \eqref{cL} is a $*$-linear map, the Choi matrix $\BL$ is Hermitian while the matricization $L$ satisfies $\ov{L}=\fC_n L \fC_q$. It is clear how to see that $\BL$ is Hermitian, but what $\ov{L}=\fC_n L \fC_q$ means is less transparent. Decompose $L$ as
\begin{equation}\label{Ldec}
L=\left[L_{ij}\right]\in \BF^{n^2 \times q^2} \quad\mbox{with}\quad L_{ij}=\left[\ell^{ij}_{kl}\right]\in \BF^{n \times q}.
\end{equation}
Using Theorem \ref{P:Herm} it follows that $\cL$ being $*$-linear is the same as
\begin{equation}\label{Lherm}
\ell^{ij}_{kl}=\ov{\ell}^{kl}_{ij},\quad 1 \le i,k\le n, \quad 1 \le j,l \le q.
\end{equation}
One way of interpreting this characterization is that ``structural properties of $L$ as a block matrix reoccur at the level of the blocks.'' We illustrate this by considering several examples in the next subsection. In Subsection \ref{SubS:HillMat} we consider the Hill matrix for the special case where the matrices $L_1,\ldots,L_m$ are chosen among the blocks $L_{ij}$, and revisit the examples in this context.

\subsection{Structural properties of $L$}

We shall present two general lemmas, each followed by a list of structures that are repeated at the level of the blocks once present at the block level, and conversely, as a consequence of the result. These lists just serve a an illustration and are by no means exhaustive.

\begin{lemma} \label{zero structure}
Assume $\cL$ in \eqref{cL} is $*$-linear and decompose $L$ defined by \eqref{Ldef} as in \eqref{Ldec}.  Let $i \in \{1,\ldots,n\}$ and $j \in\{1,\ldots,q\}$. Then $L_{ij}=0$ if and only if $(i,j)$-th entry in $L_{kl}$ is zero for each $k\in \{1,\ldots,n\}$ and $l \in \{1,\ldots,q\}$. In particular, for any subset $C\subset \{1,\ldots,n\} \times \{1,\ldots,q\}$  we have $L_{ij}=0$ for all $(i,j)\in C$ if and only if for each block $L_{kl}$, $k\in\{1,\ldots,n\}$ and $l \in \{1,\ldots, q\}$, the entries with indices corresponding to the elements in $C$ are zero.
\end{lemma}

\begin{proof}[\bf Proof]
Assume $L_{ij}=0$. Then $0=\ell_{kl}^{ij}=\ov{\ell}_{ij}^{kl}$ for all $k\in\{1,\ldots,n\}$ and $l \in \{1,\ldots, q\}$. Thus for each block $L_{kl}$, $k\in\{1,\ldots,n\}$ and $l \in \{1,\ldots, q\}$, the $(i,j)$-th entry is zero. Similarly, suppose $\ell_{ij}^{kl}=0$ for $k\in\{1,\ldots,n\}$ and $l \in \{1,\ldots, q\}$. Then $0=\ov{L}_{ij}=L_{ij}$.
\end{proof}

As a consequence we have the following structures that are present at block matrix level precisely when they are present at the level of all the blocks.
\begin{itemize}
\item[(i)] Diagonal: $L_{ij}=0$ for all $i\neq j$ $\Longleftrightarrow$ $\ell^{ij}_{kl}=0$ for all $k\neq l$ and all $i,j$.

\item[(ii)] Lower triangular: $L_{ij}=0$ for all $i<j$  $\Longleftrightarrow$ $\ell^{ij}_{kl}=0$ for all $k < l$ and all $i,j$.

\item[(ii)] Upper triangular: $L_{ij}=0$ for all $i>j$  $\Longleftrightarrow$ $\ell^{ij}_{kl}=0$ for all $k > l$ and all $i,j$.

\item[(iii)] $d$-Band matrices: $L_{ij}=0$ for all $|i-j|<d$  $\Longleftrightarrow$ $\ell^{ij}_{kl}=0$ for all $|k-l|<d$ and all $i,j$.

\item[(iv)] Hollow matrices: $L_{ii}=0$ for all $i$ $\Longleftrightarrow$ $\ell^{ij}_{kk}=0$ for all $k$ and all $i,j$.

\end{itemize}

Many variations on this can be made.

\begin{lemma}\label{repeated blocks structure}
Assume $\cL$ in \eqref{cL} is $*$-linear and decompose $L$ defined by \eqref{Ldef} as in \eqref{Ldec}.   Let $i,k \in \{1,\ldots,n\}$ and $j,l \in\{1,\ldots,q\}$. Then $L_{ij}=L_{kl}$ if and only if for all $r\in \{1,\ldots,n\}$ and $s \in \{1,\ldots,q\}$ the $(i,j)$-th entry and the $(k,l)$-th entry in $L_{rs}$ are equal.
\end{lemma}

\begin{proof}[\bf Proof]
Assume $L_{ij}=L_{kl}.$ Then $\ov{\ell}_{ij}^{rs}=\ell_{rs}^{ij}=\ell_{rs}^{kl}=\ov{\ell}_{kl}^{rs}$ for all $r \in \{1,\ldots,n\}$ and $s \in \{1,\ldots,q\}.$ Thus each $\ov{L}_{rs}$ (and hence $L_{rs}$) has the $(i,j)$-th entry equal to the $(k,l)$-th entry. For the backward direction, assume $\ell^{rs}_{ij}=\ell^{rs}_{kl}$. Using the $*$-linearity of $\cL$, this is equivalent to $\ov{\ell}^{ij}_{rs}=\ov{\ell}_{rs}^{kl}$ for all $i,j,k,l.$ Hence each entry in $\ov{L}_{ij}$ is equal to the corresponding entry in $\ov{L}_{kl}$. Thus $L_{ij}=L_{kl}.$
\end{proof}

A few more structures that appear at block level provided they appear in all the blocks, and conversely, are listed next. Here we assume $n=q$, which is required for some, but not all structures.

\begin{itemize}
\item[(i)] Toeplitz matrices: $L_{ij}=L_{kl}$ whenever $i-j=k-l$ $\Longleftrightarrow$ $\ell^{rs}_{kl}=\ell^{rs}_{ij}$ whenever $i-j=k-l$ for all $r,s$.

\item[(ii)] Hankel matrices: $L_{ij}=L_{kl}$ whenever $i+j=k+l$ $\Longleftrightarrow$ $\ell^{rs}_{kl}=\ell^{rs}_{ij}$ whenever $i+j=k+l$ for all $r,s$.

\item[(iii)] Circulant matrices: $L_{ij}=L_{kl}$ whenever $i-j\equiv k-l$ mod $n$ $\Longleftrightarrow$ $\ell^{rs}_{kl}=\ell^{rs}_{ij}$ whenever $i-j\equiv k-l$ mod $n$ for all $r,s$.

\item[(iv)] Centrosymmetric matrices: $L_{ij}=L_{n-i+1\, n-j+1}$ for all $i,j$ $\Longleftrightarrow$ $\ell^{ij}_{kl}=\ell^{ij}_{n-k+1\, n-l+1}$ for all $i,j,k,l$.

\item[(v)] Symmetric matrices: $L_{ij}=L_{ji}$ for all $i,j$ $\Longleftrightarrow$ $L_{ij}=L_{ij}^T$ for all $i,j$.

\item[(vi)] Hermitian matrices: $L_{ij}=\ov{L}_{ji}$ for all $i,j$ $\Longleftrightarrow$ $L_{ij}=L_{ij}^*$ for all $i,j$.

%
%
%

\end{itemize}

\begin{lemma}\label{L:Ortho Structure}
Assume $\cL$ in \eqref{cL} is $*$-linear and decompose $L$ defined by \eqref{Ldef} as in \eqref{Ldec}.
Let $i,k \in \{1,\ldots,n\}$ and $j,l \in\{1,\ldots,q\}$. Then $L_{ij}$ and $L_{kl}$ are orthogonal with respect to the trace inner product if and only if the matrices $\left[\ell_{ij}^{rs}\right]$ and $\left[\ell_{kl}^{rs}\right]$ in $\BF^{n \times q}$ are orthogonal with respect to the trace inner product.
\end{lemma}

\begin{proof}[\bf Proof]
The fact that $\cL$ is $*$-linear implies that
\begin{align*}
\inn{L_{ij}}{L_{kl}}_{\BF^{n \times q}} & = \vect\left(L_{kl}\right)^* \vect\left(L_{ij}\right)= \sum_{r,s} \ov{\ell}_{rs}^{kl} \ell_{rs}^{ij}
= \sum_{r,s} \ell^{rs}_{kl} \ov{\ell}^{rs}_{ij} = \sum_{r,s} \ov{\ell}^{rs}_{ij} \ell^{rs}_{kl}\\
&=  \inn{\left[\ell_{ij}^{rs}\right]}{\left[\ell_{kl}^{rs}\right]}_{\BF^{n \times q}},
\end{align*}
from which the result follows directly.
\end{proof}

\subsection{The Hill matrix}\label{SubS:HillMat}
The Hill matrix $\BH$ of $\cL$ is determined by choosing matrices $L_1,\ldots,L_m$ in $\BF^{n\times q}$ so that the span of $L_1,\ldots,L_m$ coincides with the span of the blocks $L_{ij}$ in \eqref{Ldec}. It is always possible to take for $L_1,\ldots,L_m$ $m$ linearly independent matrices among the blocks $L_{ij}$:
\begin{equation}\label{LkLijChoice}
\begin{aligned}
L_k=L_{i_k j_k},\quad k=1,\ldots,m,&\quad i_k\in\{1,\ldots,n\},\, j_k\in\{1,\ldots,q\},\\
L_1,\ldots,L_m &\mbox{ linearly independent.}
\end{aligned}
\end{equation}
In that case, the Hill matrix can be described in terms of the entries of $L$.

\begin{lemma}\label{L:HillMatSC}
Assume $\cL$ in \eqref{cL} is $*$-linear. Decompose the matricization $L$ as in \eqref{Ldec} and select $L_1,\ldots,L_m$ as in \eqref{LkLijChoice}. Then the Hill matrix $\BH$ in \eqref{Hill} determined by $\cL$ and $L_1,\ldots,L_m$ is given by
\[
\BH=\left[\ell_{i_l j_l}^{i_k j_k}\right]_{k,l=1}^m.
\]
\end{lemma}

\begin{proof}[\bf Proof]
For this choice of $L_1,\ldots,L_m$, for $\beta_{ij}^k$ in \eqref{LkLij} we have $\beta_{ij}^k=0$ for all $(i,j)\ne(i_k,j_k)$ and $\beta_{i_kj_k}^k=1$. Thus
\begin{equation*}
B_k=\begin{bmatrix}\beta_{ij}^k \end{bmatrix}
=\cE_{i_kj_k}^{(n,q)} \quad \text{for all} \quad  1 \le k \le m.
\end{equation*}
Now
\begin{align*}
\BH
&=\begin{bmatrix}\OneVec_n^*\left(B_k \circ \overline{L}_l\right)\OneVec_q \end{bmatrix}
=\begin{bmatrix}\OneVec_n^*\left(\cE_{i_kj_k}^{(n,q)} \circ \ov{L}_{i_l j_l}\right)\OneVec_q \end{bmatrix}
=\begin{bmatrix}\ell_{i_l j_l}^{i_k j_k}\end{bmatrix}.\qedhere
\end{align*}
\end{proof}

Note that the Hill matrix $\BH$ is of size $m \times m$, hence consists of $m^2$ scalar entries. On the other hand, the linearly independent matrices $L_1,\ldots,L_m$ are of size $n \times q$, hence together they consist of $mnq$ scalar entries; often this will be more than $m^2$. That we do not loose any information is a result from the relations in \eqref{Lherm}. We illustrate this in two examples based on Lemmas \ref{zero structure}  and \ref{repeated blocks structure}.

\begin{example}
Consider the case where $L \in \BF^{n^2 \times q^2}$ is a block matrix with blocks $L_{ij}\in \BF^{n \times q}$ so that there are $m$ linearly independent blocks $L_k=L_{i_k j_k}$, $k=1,\ldots,m$, while all other blocks are zero, i.e., $L_{ij}=0$ if $(i,j)\neq (i_l,j_l)$ for all $l$. Write $L_k=\left[\ell^k_{rs}\right]$ and $L_{ij}=\left[\ell^{ij}_{rs}\right]$. Hence, by Lemma \ref{L:HillMatSC} the Hill matrix associated with this choice of $L_1,\ldots,L_m$ is given by $\BH=\left[\ell^k_{i_l,j_l}\right]\in \BF^{m \times m}$. Now using Lemma \ref{zero structure} and the fact that $L_{ij}=0$ if $(i,j)\neq (i_l,j_l)$ for all $l$, it follows that $\ell^k_{ij}=\ell^{i_k j_k}_{ij}=0$ whenever $(i,j)\neq (i_l,j_l)$ for all $l$. In particular, all non-zero entries of $L$ are contained in $\BH$.
\end{example}

\begin{example}
Consider the case where $L \in \BF^{n^2 \times q^2}$ is a block matrix with blocks $L_{ij}\in \BF^{n \times q}$ so that there are $m$ linearly independent blocks $L_k=L_{i_k j_k}$, $k=1,\ldots,m$, while all other blocks are equal to one of the matrices $L_1,\ldots,L_m$, i.e., for all $i,j$ we have $L_{ij}=L_{k_{ij}}$ for some $0\leq k_{ij}\leq m$. With $L_k$ and $L_{ij}$ decomposed as in the previuos example, again we get $\BH=\left[\ell^k_{i_l,j_l}\right]=\left[\ell^{i_k j_k}_{i_l,j_l}\right]\in \BF^{m \times m}$. In this case, since $L_{rs}=L_{k_{rs}}=L_{i_{k_{rs}} j_{k_{sr}}}$, we know from Lemma \ref{repeated blocks structure} that for $k=1,\ldots,m$ we have $\ell^{k}_{rs}=\ell^{i_k j_k}_{rs}= \ell^{i_k j_k}_{i_{k_{rs}} j_{k_{rs}}}= \ell^{k}_{i_{k_{rs}} j_{k_{rs}}}$. Thus also when $(r,s)\neq (i_l j_l)$ for all $l$, then the number $\ell^{k}_{rs}$ still appears in $\BH$.
\end{example}

\paragraph{\bf Acknowledgments}
This work is based on research supported in part by the National Research Foundation of South Africa (NRF) and the DSI-NRF Centre of Excellence in Mathematical and Statistical Sciences (CoE-MaSS). Any opinion, finding and conclusion or recommendation expressed in this material is that of the authors and the NRF and CoE-MaSS do not accept any liability in this regard.

\end{document}